\documentclass{cmslatex}
\usepackage[paperwidth=7in, paperheight=10in, margin=.875in]{geometry}
 \usepackage[backref,colorlinks,linkcolor=red,anchorcolor=green,citecolor=blue]{hyperref}
\usepackage{amsfonts,amssymb}
\usepackage{amsmath}
\usepackage{tcolorbox}
\usepackage{graphicx}
\usepackage{cite}
\usepackage{enumerate}
\renewcommand{\theequation}{\arabic{section}.\arabic{equation}}

   \allowdisplaybreaks
\begin{document}
 \title{Corrector homogenization estimates for a non-stationary
 	Stokes-Nernst-Planck-Poisson system in perforated domains
 	\thanks{Received date, and accepted date (The correct dates will be entered by the editor).}}


          \author{Vo Anh Khoa\thanks{Author for correspondence. Mathematics and Computer Science Division,
          		Gran Sasso Science Institute, L'Aquila, Italy, (khoa.vo@gssi.infn.it,
          		vakhoa.hcmus@gmail.com).}
          \and Adrian Muntean\thanks{Department of Mathematics and Computer Science, Karlstad University,
          	Sweden, (adrian.muntean@kau.se).}}

         \pagestyle{myheadings} \markboth{Correctors for a Stokes-Nernst-Planck-Poisson system}{V.A. Khoa and A. Muntean} \maketitle

          \begin{abstract}
               We consider a non-stationary Stokes-Nernst-Planck-Poisson
               system posed in perforated domains. Our aim is to justify rigorously the homogenization limit for the
               upscaled system derived by means of two-scale convergence in \cite{RMK12}. 
               In other words, we wish to obtain the so-called corrector homogenization
               estimates that specify the error obtained when upscaling the microscopic
               equations. Essentially, we control in terms of suitable norms differences
               between the micro- and macro-concentrations and between the corresponding
               micro- and macro-concentration gradients. The major challenges
               that we face are the coupled flux structure of the
               system, the nonlinear drift terms and the presence of the microstructures.
               Employing various energy-like estimates, we discuss several scalings choices
               and boundary conditions.
          \end{abstract}


\section{Introduction}\label{intro}
Colloidal dynamics is a relevant research topic of interest from
both theoretical perspectives and modern industrial
applications. Relevant technological applications include oil recovery
and transport \cite{TM99}, drug-delivery design \cite{ML09}, motion of micro-organisms
in biological suspensions \cite{Earl06}, harvesting energy via solar cells \cite{Boda}, and also, sol-gel synthesis
\cite{BS90}. Typically, they all involve different phases of dispersed media (solid morphologies), which
resemble  at least remotely to homogeneous domains paved with arrays of contrasting microstructures that are distributed periodically. Mathematically, the interplay between populations of  colloidal particles lead to work in the multiscale analysis of PDEs especially what concerns the Smoluchowski coagulation-fragmentation system and the Stokes-Nernst-Planck-Poisson system, which is our target here.

It is well known (cf. \cite{Ven89}, e.g.)  that
many particles in colloidal chemistry  are able to carry
electrical charges (positive or negative) and, in some circumstances,  they can be described using intensive quantities like the number density or ions concentration, say  $c_{\varepsilon}^{\pm}$. Following
\cite{EGJW98}, we consider such concentrations $c_{\varepsilon}^{\pm}$
of electrically charged colloidal particles to be involved as unknowns in the Nernst-Planck equations.
These equations model the diffusion, deposition, convection
and electrostatic interaction within a porous medium. The associated electrostatic
potential, called here $\Phi_{\varepsilon}$, is usually determined by a Poisson equation
linearly coupled with the densities of charged species, describing the electric
field formation inside the heterogeneous domain. Colloidal particles are always immersed in a background fluid. Here, we assume that the fluid velocity $v_{\varepsilon}$
fulfills a suitable variant of the Stokes equations.

It is the aim of this paper to explore mathematically the upscaling of such non-stationary Stokes-Nernst-Planck-Poisson
(SNPP) systems posed in a porous medium $\Omega^{\varepsilon}\subset\mathbb{R}^{d}$, where $\varepsilon\in\left(0,1\right)$ represents the scale
parameter relative to the perforation (pore sizes) of the domain. To be more precise, we wish to justify the homogenization asymptotics for a class of SNPP systems developed by the group of Prof. P. Knabner in Erlangen, Germany,  that fit well to the motion of charged colloidal particles through saturated soils.

As starting point of the discussion, we consider the following microscopic Stokes-Nernst-Planck-Poisson (SNPP) system:
\begin{align}
	& -\varepsilon^{2}\Delta v_{\varepsilon}+\nabla p_{\varepsilon}=-\varepsilon^{\beta}\left(c_{\varepsilon}^{+}-c_{\varepsilon}^{-}\right)\nabla\Phi_{\varepsilon}\quad\text{in}\quad Q_{T}^{\varepsilon}:=\left(0,T\right)\times\Omega^{\varepsilon}, \label{eq:haha1} \\
	& \nabla\cdot v_{\varepsilon}=0\quad \text{in}\quad Q_{T}^{\varepsilon},\\
	& v_{\varepsilon}=0\quad \text{on}\quad \left(0,T\right)\times\left(\Gamma^{\varepsilon}\cup\partial\Omega\right), \label{eq:haha-1-1}\\
	& -\varepsilon^{\alpha}\Delta\Phi_{\varepsilon}=c_{\varepsilon}^{+}-c_{\varepsilon}^{-}\quad \text{in}\quad Q_{T}^{\varepsilon},
	\label{eq:haha4}\\
	& \varepsilon^{\alpha}\nabla\Phi_{\varepsilon}\cdot\text{n}=0\quad\text{on}\quad \left(0,T\right)\times\partial\Omega,
	\label{eq:haha7}\\
	& \partial_{t}c_{\varepsilon}^{\pm}+\nabla\cdot\left(v_{\varepsilon}c_{\varepsilon}^{\pm}-\nabla c_{\varepsilon}^{\pm}\mp\varepsilon^{\gamma}c_{\varepsilon}^{\pm}\nabla\Phi_{\varepsilon}\right)=R_{\varepsilon}^{\pm}\left(c_{\varepsilon}^{+},c_{\varepsilon}^{-}\right)\quad \text{in}\quad Q_{T}^{\varepsilon},
	\label{eq:haha8}\\
	& -\left(v_{\varepsilon}c_{\varepsilon}^{\pm}-\nabla c_{\varepsilon}^{\pm}\mp\varepsilon^{\gamma}c_{\varepsilon}^{\pm}\nabla\Phi_{\varepsilon}\right)\cdot\text{n =0}\quad \text{on}\quad \left(0,T\right)\times\left(\Gamma^{\varepsilon}\cup\partial\Omega\right),\\
	& c_{\varepsilon}^{\pm}=c^{\pm,0}\quad \text{in}\quad \left\{ t=0\right\} \times\Omega^{\varepsilon}.\label{eq:haha10}
\end{align}

We refer to \eqref{eq:haha1}-\eqref{eq:haha10} as $\left(P^{\varepsilon}\right)$. The system \eqref{eq:haha1}-\eqref{eq:haha10} is endowed either with
\begin{align}
\varepsilon^{\alpha}\nabla\Phi_{\varepsilon}\cdot\text{n}=\varepsilon\sigma\quad \text{on}\quad \left(0,T\right)\times\Gamma_{N}^{\varepsilon},\label{eq:boundN}
\end{align}
or with
\begin{align}
\Phi_{\varepsilon}=\Phi_{D}\quad \text{on}\quad \left(0,T\right)\times\Gamma_{D}^{\varepsilon},\label{eq:boundD}
\end{align}

We deliberately use variable scaling parameters $\alpha,\beta,\gamma$ for the ratio of the magnitudes of differently incorporated physical processes to weigh the effect a certain heterogeneity (morphology) has on effective transport coefficients.

A few additional remarks are in order:  The background fluid (solvent) is assumed to be isothermal, incompressible and electrically neutral. The movement of this liquid at low Reynolds numbers decides the momentum equation behind our Stokes flow (see in \eqref{eq:haha1}-\eqref{eq:haha-1-1}). The Stokes equation further couples to the mass balance equations of the involved colloidal species as described by the  Nernst-Planck equations in \eqref{eq:haha8}-\eqref{eq:haha10}. The initial charged densities $c^{\pm,0}$ are present cf. \eqref{eq:haha10}. The
Poisson-type equation points out an induced electric field acting on the liquid as well as on the charges carried by the colloidal  species (see in \eqref{eq:haha4}-\eqref{eq:haha7}). The surface charge density $\sigma$ of the porous medium is prescribed as in \eqref{eq:boundN}.

Although it can in principle introduce a boundary layer potentially interacting with the homogenization asymptotics, the magnitude of the $\zeta$-potential $\Phi_{D}$ in \eqref{eq:boundD} does
not influence our theoretical results.  Here, it only indicates the degree
of electrostatic repulsion between charged colloidal particles within
a dispersion. In fact, experiments provide that colloids with high
$\zeta$-potential (i.e. $\Phi_{D}\gg1$ or $\Phi_{D}\ll-1$) are
electrically stabilized while with low $\zeta$-potential, they
tend to coagulate or flocculate rapidly (see e.g. \cite{GK99,NT-A04} for
a detailed calculation).

\begin{table}
	\begin{centering}
		\begin{tabular}{|c|c|}
			\hline 
			$v_{\varepsilon}:Q_{T}^{\varepsilon} \to \mathbb{R}$ & velocity\tabularnewline
			\hline 
			$p_{\varepsilon}:Q_{T}^{\varepsilon} \to \mathbb{R}$ & pressure\tabularnewline
			\hline 
			$\Phi_{\varepsilon}: Q_{T}^{\varepsilon} \to \mathbb{R}$ & electrostatic potential\tabularnewline
			\hline 
			$c_{\varepsilon}^{\pm}:Q_{T}^{\varepsilon} \to \mathbb{R}$ & number densities\tabularnewline
			\hline
			$c^{\pm,0}:\Omega^{\varepsilon} \to \mathbb{R}$ & initial charged densities\tabularnewline
			\hline 
			$\sigma\in\mathbb{R}$ & surface charge density\tabularnewline
			\hline 
			$\Phi_{D}\in\mathbb{R}$ & $\zeta$-potential\tabularnewline
			\hline 
			$R_{\varepsilon}^{\pm}:\mathbb{R}^{2}\to\mathbb{R}$ & reaction rates\tabularnewline
			\hline 
			$\alpha,\beta,\gamma\in\mathbb{R}$ & variable choices of scalings\tabularnewline
			\hline 
		\end{tabular}
		\par\end{centering}
	\caption{Physical unknowns and parameters arising in the microscopic problem $\left(P^{\varepsilon}\right)$.\label{tab:1-1}}
\end{table}

Specific scenarios for averaging Poisson-Nernst-Planck (PNP) systems as well as Stokes-Nernst-Planck-Poisson (SNPP) systems were discussed in a number of recent papers; see e.g. \cite{SB15,Schmuck11,Frank13,FK16,GM16,GM14}.
The SNPP-type models are more difficult to handle mathematically mostly because of the oscillations introduced by the presence of the Stokes flow.   The SNPP systems shown in \cite{RMK12,FRK11} are endowed with several
scaling choices to cover various types of SNPP systems including
Schmuck's work cf. \cite{Schmuck11} and the study of a stationary
and linearized SNPP system by Allaire et al. cf. \cite{AMP10}. As main
results, the global weak solvability of the respective models as well as their periodic
homogenization limit procedures were obtained. We refer to reader to the {\em lit. cit.} also for the  precise structure of
the associated effective transport tensor parameters and upscaled equations. It is worth also mentioning that sometimes, like e.g. in \cite{SB15,Schmuck11,Schmuck12}, a classification of the upscaling results is done depending on the choice of boundary conditions for the Poisson equation.

The main theme of this paper is the derivation of  corrector estimates
quantifying the convergence rate of the periodic homogenization limit process leading to upscaled SNPP systems. This should be seen as a quantitative check of the quality of the two-scale averaging procedure.
Getting grip on corrector estimates
is a needed step in designing convergent multiscale finite element
methods (see, e.g. \cite{Hou1999}) and can play an important role also in studying multiscale inverse problems.


Our main results
are reported in Theorem \ref{thm:main1-6} in and
Theorem \ref{thm:main2-6}.  Here both the Neumann and Dirichlet boundary data for the electrostatic potential are considered. The two types of boundary conditions for the electrostatic potential will lead to different structures of the upscaled systems, and hence, also the structure of the correctors will be different.  To obtain these
corrector estimates, we rely on the energy method combined with integral estimates for periodically oscillating functions as well as with  appropriate macroscopic reconstructions, regularity
results on limit and cell functions as well as the smoothness assumptions
for the microscopic boundaries and data. It is worth mentioning that the corrector
estimate for the closest model to ours, i.e. for the PNP
equations in \cite[Theorem 2.3]{Schmuck12}, reveals already a class of possible assumptions  on
the cell functions (taken in $W^{1,\infty}$) as well as on the smoothness of the interior and
exterior boundaries (taken in $C^{\infty}$). Also, we borrowed ideas from both linear elliptic theory \cite{ADN59} as well as from the techniques behind the previously obtained corrector estimates \cite{CP99,KM16,KM17,Khoa17} for periodically perforated media. Concerning the locally periodic case, we refer the reader to \cite{Tycho} and references cited therein or to Zhang et al. \cite{ZBXY17}. In the latter paper, the authors have studied the homogenization of a steady reaction-diffusion system in a chemical vapor infiltration (CVI) process and have also deduced the convergence rate for the homogenization limit.

The reader should bear in mind that our way of deriving  corrector estimates does not extend to the stochastic homogenization setting, but can cover, involving only minimal technical modifications, the locally periodic homogenization setting.

The corrector estimates we claim are the following:

	\textbf{Case 1:} If the electrostatic potential $\Phi_{\varepsilon}$
satisfies the homogeneous Neumann boundary condition, then it holds
\begin{tcolorbox}		
	\begin{align}
	& \left\Vert \tilde{\Phi}_{\varepsilon}-\tilde{\Phi}_{0}^{\varepsilon}\right\Vert _{L^{2}\left(\left(0,T\right)\times\Omega^{\varepsilon}\right)}+\left\Vert c_{\varepsilon}^{\pm}-c_{0}^{\pm,\varepsilon}\right\Vert _{L^{2}\left(\left(0,T\right)\times\Omega^{\varepsilon}\right)}\nonumber\\
	& +\left\Vert \nabla\left(\tilde{\Phi}_{\varepsilon}-\tilde{\Phi}_{1}^{\varepsilon}\right)\right\Vert _{\left[L^{2}\left(\left(0,T\right)\times\Omega^{\varepsilon}\right)\right]^{d}}\le C\max\left\{ \varepsilon^{\frac{1}{2}},\varepsilon^{\frac{\mu}{2}}\right\} ,\\
	& \left\Vert \nabla\left(c_{\varepsilon}^{\pm}-c_{1}^{\pm,\varepsilon}\right)\right\Vert _{\left[L^{2}\left(\left(0,T\right)\times\Omega^{\varepsilon}\right)\right]^{d}}\le C\max\left\{ \varepsilon^{\frac{1}{4}},\varepsilon^{\frac{\mu}{2}}\right\} ,\\
	& \left\Vert v_{\varepsilon}-\left|Y_{l}\right|^{-1}\mathbb{D}v_{0}^{\varepsilon}-\varepsilon\left|Y_{l}\right|^{-1}\mathbb{D}v_{1}^{\varepsilon}\right\Vert _{\left[L^{2}\left(\left(0,T\right)\times\Omega^{\varepsilon}\right)\right]^{d}}\nonumber\\
	& +\left\Vert p_{\varepsilon}-p_{0}\right\Vert _{L^{2}\left(\Omega\right)/\mathbb{R}}\le C\left(\max\left\{ \varepsilon^{\frac{1}{2}},\varepsilon^{\frac{\mu}{2}}\right\} +\varepsilon^{\frac{\lambda}{2}}+\varepsilon^{1-\frac{3\lambda}{2}}+\varepsilon^{\frac{1}{2}-\lambda}\right),
	\end{align}
	where $\mu \in \mathbb{R}_{+}$ and $\lambda \in (0,1)$.
	
\end{tcolorbox}

\textbf{Case 2:} If the electrostatic potential $\Phi_{\varepsilon}$
satisfies the homogeneous Dirichlet boundary condition, then it holds
\begin{tcolorbox}
\begin{align}
& \left\Vert \tilde{\Phi}_{\varepsilon}-\tilde{\Phi}_{0}^{\varepsilon}\right\Vert _{L^{2}\left(\left(0,T\right)\times\Omega^{\varepsilon}\right)}+\left\Vert c_{\varepsilon}^{\pm}-c_{0}^{\pm,\varepsilon}\right\Vert _{L^{2}\left(\left(0,T\right)\times\Omega^{\varepsilon}\right)}\nonumber \\
& +\left\Vert \nabla\left(\tilde{\Phi}_{\varepsilon}-\tilde{\Phi}_{0}^{\varepsilon}\right)\right\Vert _{\left[L^{2}\left(\left(0,T\right)\times\Omega^{\varepsilon}\right)\right]^{d}}+\left\Vert \nabla\left(c_{\varepsilon}^{\pm}-c_{1}^{\pm,\varepsilon}\right)\right\Vert _{\left[L^{2}\left(\left(0,T\right)\times\Omega^{\varepsilon}\right)\right]^{d}}\le C\max\left\{ \varepsilon^{\frac{1}{2}},\varepsilon^{\frac{\mu}{2}}\right\} ,\\
& \left\Vert v_{\varepsilon}-\left|Y_{l}\right|^{-1}\mathbb{D}v_{0}^{\varepsilon}-\varepsilon\left|Y_{l}\right|^{-1}\mathbb{D}v_{1}^{\varepsilon}\right\Vert _{\left[L^{2}\left(\left(0,T\right)\times\Omega^{\varepsilon}\right)\right]^{d}}\nonumber\\
& +\left\Vert p_{\varepsilon}-p_{0}\right\Vert _{L^{2}\left(\Omega\right)/\mathbb{R}}\le C\left(\max\left\{ \varepsilon^{\frac{1}{2}},\varepsilon^{\frac{\mu}{2}}\right\} +\varepsilon^{\frac{\lambda}{2}}+\varepsilon^{1-\frac{3\lambda}{2}}+\varepsilon^{\frac{1}{2}-\lambda}\right).
\end{align}
\end{tcolorbox}

The paper is organized as follows. In Section \ref{sec:Setting-of-the}, the geometry
of our perforated domains is introduced together with some notation and conventions.
The list of assumptions on the data is also reported here. In the second
part of the section, we present the classical concepts of the two-scale
convergence on domains and on surfaces and then provide the weak and strong formulations of all systems of PDEs mentioned in this framework (including the microscopic and
macroscopic evolution systems, the cell problems). Section \ref{mainsec} is devoted to the statement of our main results and to the corresponding proofs. The remarks from Section \ref{finalsec} conclude  the paper.


\section{Technical preliminaries\label{sec:Setting-of-the}}
\subsection{A geometrical interpretation of porous media\label{subsec:A-geometrical-interpretation}}

Let $\Omega$ be a bounded and open domain in $\mathbb{R}^{d}$ with
$\partial\Omega\in C^{0,1}$. Without loss of generality, we assume  $\Omega$ to be the parallelepiped $\left(0,a_{1}\right)\times...\times\left(0,a_{d}\right)$
for $a_{i}>0,i\in\left\{ 1,...,d\right\} $.

Let $Y$ be the unit cell defined by
\[
Y:=\left\{ \sum_{i=1}^{d}\lambda_{i}\vec{e}_{i}:0<\lambda_{i}<1\right\} ,
\]
where $\vec{e}_{i}$ denotes the $i$th unit vector in $\mathbb{R}^{d}$.
We suppose that $Y$ consists of two open sets $Y_{l}$ and $Y_{s}$
which respectively represent the liquid part (the pore) and the solid part (the skeleton) such
that $\bar{Y}_{l}\cup\bar{Y}_{s}=\bar{Y}$ and $Y_{l}\cap Y_{s}=\emptyset,$ while $\bar{Y}_{l}\cap\bar{Y}_{s}=\Gamma$ 
has a non-zero $\left(d-1\right)$-dimensional Hausdorff measure. Additionally,
we do not allow the solid part $Y_{s}$ to touch the outer boundary
$\partial Y$ of the unit cell. As a consequence, the fluid part is
connected (see Figure \ref{fig:1-666}).

Let $Z\subset\mathbb{R}^{d}$ be a hypercube. For $X\subset Z$ we
denote by $X^{k}$ the shifted subset
\[
X^{k}:=X+\sum_{i=1}^{d}k_{i}\vec{e}_{i},
\]
where $k=\left(k_{1},...,k_{d}\right)\in\mathbb{Z}^{d}$ is a vector
of indices.

Let $\varepsilon>0$ be a given scale factor. We assume that $\Omega$
is completely covered by a regular array
of $\varepsilon$-scaled shifted cells. In porous media terminology, the solid part/pore skeleton
is defined as the union of the cell regions $\varepsilon Y_{s}^{k}$,
i.e.
\[
\Omega_{0}^{\varepsilon}:=\bigcup_{k\in\mathbb{Z}^{d}}\varepsilon Y_{s}^{k},
\]
while the fluid part, which is filling up the total space, is represented by
\[
\Omega^{\varepsilon}:=\bigcup_{k\in\mathbb{Z}^{d}}\varepsilon Y_{l}^{k}.
\]

We denote the total pore surface of the skeleton by $\Gamma^{\varepsilon}:=\partial\Omega_{0}^{\varepsilon}$. This description indicates that the porous medium we have in mind is saturated with the fluid.

Note that we use the subscripts $N$ and $D$ in (\ref{eq:boundN})-(\ref{eq:boundD})
to distinguish, respectively, the case when the Neumann and Dirichlet
conditions are applied across the pore surface. Furthermore, the
assumption $\partial\Omega\cap\Gamma^{\varepsilon}=\emptyset$ holds.

In Figure \ref{fig:1-666}, we show an admissible geometry mimicking a porous medium
with periodic microstructures. We let
$\mbox{n}_{\varepsilon}:=\left(n_{1},...,n_{d}\right)$ be the unit outward normal
vector on the boundary $\Gamma^{\varepsilon}$. The representation
of the periodic geometries is in line with the descriptions from \cite{HJ91,KM16,RMK12}
and the references cited therein.

\begin{figure}
	\begin{centering}
		\includegraphics[scale=0.65]{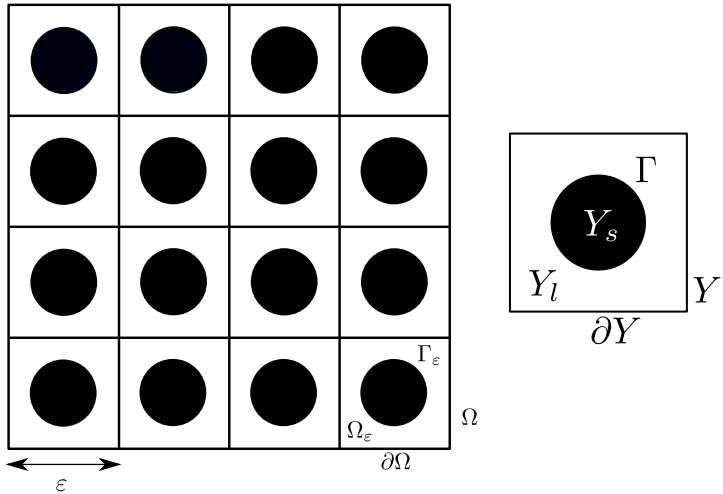}
		\par\end{centering}
	\caption{An admissible perforated domain. The perforations are referred here as microstructures.\label{fig:1-666}}
\end{figure}

We denote by $x\in\Omega^{\varepsilon}$ the macroscopic variable
and by $y=x/\varepsilon$ the microscopic variable representing fast
variations at the microscopic geometry. In the following, the upper
index $\varepsilon$ thus denotes the corresponding quantity evaluated
at $y=x/\varepsilon$. Suppose that our total pore space $\Omega^{\varepsilon}$
is bounded, connected and possesses $C^{0,1}$-boundary.

In the sequel, all the constants $C$ are independent of the homogenization
parameter $\varepsilon$, but their precise values may differ from
line to line and may change even within a single chain of estimates.
Throughout this paper, we use the superscript $\varepsilon$ to emphasize
the dependence of the material on the heterogeneity characterized
by the homogenization parameter.  In the following, we use $dS_{\varepsilon}$ to indicate the surface measure of oscillating surfaces (boundary of microstructures).  In addition, depending on the context, by $\left|\cdot\right|$ we denote either the volume measure of a domain or the absolute value of a function domain.

When writing the superscript $\pm$ or $\mp$ in e.g. $c_{\varepsilon}^{\pm}$, we mean both the positive  $c_{\varepsilon}^{+}$ and negative densities $c_{\varepsilon}^{-}$.

Due to our choice of microstructures, the interior extension from
$H^{1}\left(\Omega^{\varepsilon}\right)$ into $H^{1}\left(\Omega\right)$
exists and the extension constant is independent of $\varepsilon$
(see \cite[Lemma 5]{HJ91}).

\subsection{Assumptions on the data}
To ensure the weak solvability of our SNPP system, we need essentially several assumptions on the involved data and parameters.

$\left(\mbox{A}_{1}\right)$ The initial data of charged densities
are non-negative and bounded independently of $\varepsilon$, i.e. there exists an $\varepsilon$-independent constant $C_0>0$ such that
\[
0\le c^{\pm,0}\left(x\right)\le C_{0}\quad\text{for a.e. }\;x\in\Omega.
\]

$\left(\mbox{A}_{2}\right)$ The initial data of charged densities satisfy the compatibility
condition:
\[
\int_{\Omega^{\varepsilon}}\left(c^{+,0}-c^{-,0}\right)dx=\int_{\Gamma^{\varepsilon}}\sigma dS_{\varepsilon}.
\]

$\left(\mbox{A}_{3}\right)$ The chemical reaction rates are structured
as $R_{\varepsilon}^{\pm}\left(c_{\varepsilon}^{+},c_{\varepsilon}^{-}\right)=\mp\left(c_{\varepsilon}^{+}-c_{\varepsilon}^{-}\right)$.

$\left(\mbox{A}_{4}\right)$ The surface charge density $\sigma$
and the $\zeta$-potential $\Phi_{D}$ are constants.

$\left(\mbox{A}_{5}\right)$ The electrostatic potential $\Phi_{\varepsilon}$
has zero mean value in the fluid part, i.e. it satisfies
\[
\int_{\Omega^{\varepsilon}}\Phi_{\varepsilon}dx=0.
\]

$\left(\mbox{A}_{6}\right)$ The pressure $p_{\varepsilon}$ has zero
mean value in the fluid part, i.e. it satisfies
\[
\int_{\Omega^{\varepsilon}}p_{\varepsilon}(t,x)dx=0 \quad\text{for all }\;t\ge 0.
\]

\begin{rem}
	Assumption $\left(\mbox{A}_{1}\right)$ implies that at the initial
	moment, our charged colloidal particles are either neutral or positive
	in the macroscopic domain and their maximum voltage is known.
	Based on $\left(\mbox{A}_{2}\right)$,
	if the surface charge density is static (i.e. $\sigma=0$), then we obtain
	the so-called global charge neutrality which means that the charge
	density of our colloidal particles $c_{\varepsilon}^{\pm}$ is initially
	in neutrality. This global electroneutrality condition is particularly helpful in the analysis work (well-posedness, upscaling approach and numerical scheme) of related systems as stated in e.g.  \cite{SMRB99,SB15,RMK12}. Nevertheless, it is not  used in the derivation of the corrector estimates in this work.
	Cf. $\left(\mbox{A}_{3}\right)$, the reaction rates
	are linear and ensure the conservation of mass for the concentration
	fields.
\end{rem}

\section{Weak settings of SNPP models}

\subsection{Preliminary results}

In this subsection, we present the definition of two-scale convergence
as well as related compactness arguments (cf. \cite{All92,Ngu89}). We
also recall the results on the weak solvability and periodic homogenization
of the problem $\left(P^{\varepsilon}\right)$, which are derived rigorously in
\cite{ray2013thesis,RMK12}, e.g.
\begin{defn}
	\textbf{Two-scale convergence}
	
	Let $\left(u^{\varepsilon}\right)$ be a sequence of functions in
	$L^{2}\left(\left(0,T\right)\times\Omega\right)$ with $\Omega$ being
	an open set in $\mathbb{R}^{d}$, then it two-scale converges to a
	unique function $u^{0}\in L^{2}\left(\left(0,T\right)\times\Omega\times Y\right)$,
	denoted by $u^{\varepsilon}\stackrel{2}{\rightharpoonup}u^{0}$, if
	for any $\varphi\in C_{0}^{\infty}\left(\left(0,T\right)\times\Omega;C_{\#}^{\infty}\left(Y\right)\right)$
	we have
	\[
	\lim_{\varepsilon\to0}\int_{0}^{T}\int_{\Omega}u^{\varepsilon}\left(t,x\right)\varphi\left(t,x,\frac{x}{\varepsilon}\right)dxdt=\frac{1}{\left|Y\right|}\int_{0}^{T}\int_{\Omega}\int_{Y}u^{0}\left(t,x,y\right)\varphi\left(t,x,y\right)dydxdt.
	\]
\end{defn}
\begin{thm}
	\emph{\textbf{Two-scale compactness}}
	
	\begin{itemize}
		\item Let $\left(u^{\varepsilon}\right)$ be a bounded sequence in $L^{2}\left(\left(0,T\right)\times\Omega\right)$.
		Then there exists a function $u^{0}\in L^{2}\left(\left(0,T\right)\times\Omega\times Y\right)$
		such that, up to a subsequence, $u^{\varepsilon}$ two-scale converges
		to $u^{0}$.
		\item Let $\left(u^{\varepsilon}\right)$ be a bounded sequence in $L^{2}\left(0,T;H^{1}\left(\Omega\right)\right)$,
		then up to a subsequence, we have the two-scale convergence in gradient
		$\nabla u^{\varepsilon}\stackrel{2}{\rightharpoonup}\nabla_{x}u^{0}+\nabla_{y}u^{1}$
		for $u^{0}\in L^{2}\left(\left(0,T\right)\times\Omega\times Y\right)$
		and $u^{1}\in L^{2}\left(\left(0,T\right)\times\Omega;H_{\#}^{1}\left(Y\right)/\mathbb{R}\right)$.
	\end{itemize}
\end{thm}

\begin{defn}
	\textbf{Two-scale convergence for $\varepsilon$-periodic
		hypersurfaces}
	
	Let
	$\left(u^{\varepsilon}\right)$ be a sequence of functions in $L^{2}\left(\left(0,T\right)\times\Gamma^{\varepsilon}\right)$,
	then $u^{\varepsilon}$ two-scale converges to a limit $u^{0}\in L^{2}\left(\left(0,T\right)\times\Omega\times \Gamma \right)$
	if for any $\varphi\in C_{0}^{\infty}\left(\left(0,T\right)\times\Omega;C_{\#}^{\infty}\left(\Gamma\right)\right)$
	we have
	\[
	\lim_{\varepsilon\to0}\int_{0}^{T}\int_{\Gamma^{\varepsilon}}\varepsilon u^{\varepsilon}\left(t,x\right)\varphi\left(t,x,\frac{x}{\varepsilon}\right)dS_{\varepsilon}dt=\frac{1}{\left|Y\right|}\int_{0}^{T}\int_{\Omega}\int_{\Gamma}u^{0}\left(t,x,y\right)\varphi\left(t,x,y\right)dS_{y}dxdt.
	\]	
\end{defn}
\begin{rem}
	The two-scale compactness on surfaces is the following:
	for each bounded sequence $\left(u^{\varepsilon}\right)$ in $L^{2}\left(\left(0,T\right)\times\Gamma^{\varepsilon}\right)$,
	one can extract a subsequence which two-scale converges to a limit
	$u^{0}\in L^{2}\left(\left(0,T\right)\times\Omega\times\Gamma\right)$.
	Furthermore, if $\left(u^{\varepsilon}\right)$ is bounded in $L^{\infty}\left(\left(0,T\right)\times\Gamma^{\varepsilon}\right)$,
	it then two-scale converges to a limit function $u^{0}\in L^{\infty}\left(\left(0,T\right)\times\Omega\times\Gamma\right)$.
\end{rem}
\begin{defn}
	\label{def:mathbb=00007BPeps=00007D}\textbf{Weak formulation of $({P}^{\varepsilon})$}
	
	The vector  $\left(v_{\varepsilon},p_{\varepsilon},\Phi_{\varepsilon},c_{\varepsilon}^{\pm}\right)$
	satisfying
	\[
	v_{\varepsilon}\in L^{\infty}\left(0,T;H_{0}^{1}\left(\Omega^{\varepsilon}\right)\right),p_{\varepsilon}\in L^{\infty}\left(0,T;L^{2}\left(\Omega^{\varepsilon}\right)\right),\Phi_{\varepsilon}\in L^{\infty}\left(0,T;H^{1}\left(\Omega^{\varepsilon}\right)\right),
	\]
	\[
	c_{\varepsilon}^{\pm}\in L^{\infty}\left(0,T;L^{2}\left(\Omega^{\varepsilon}\right)\right)\cap L^{2}\left(0,T;H^{1}\left(\Omega^{\varepsilon}\right)\right),\partial_{t}c_{\varepsilon}^{\pm}\in L^{2}\left(0,T;\left(H^{1}\left(\Omega^{\varepsilon}\right)\right)'\right),
	\]
	is a weak solution to $\left(P^{\varepsilon}\right)$ provided that
	\begin{align}
		& \int_{\Omega^{\varepsilon}}\left(\varepsilon^{2}\nabla v_{\varepsilon}\cdot\nabla\varphi_{1}-p_{\varepsilon}\nabla\cdot\varphi_{1}\right)dx=-\int_{\Omega^{\varepsilon}}\varepsilon^{\beta}\left(c_{\varepsilon}^{+}-c_{\varepsilon}^{-}\right)\nabla\Phi_{\varepsilon}\cdot\varphi_{1}dx,\\
		& \int_{\Omega^{\varepsilon}}v_{\varepsilon}\cdot\nabla\psi dx=0,\\
		& \int_{\Omega^{\varepsilon}}\varepsilon^{\alpha}\nabla\Phi_{\varepsilon}\cdot\nabla\varphi_{2}dx-\int_{\Gamma^{\varepsilon}}\varepsilon^{\alpha}\nabla\Phi_{\varepsilon}\cdot\text{n}\varphi_{2}dS_{\varepsilon}=\int_{\Omega^{\varepsilon}}\left(c_{\varepsilon}^{+}-c_{\varepsilon}^{-}\right)\varphi_{2}dx,\\
		& \left\langle \partial_{t}c_{\varepsilon}^{\pm},\varphi_{3}\right\rangle _{\left(H^{1}(\Omega^{\varepsilon})\right)',H^{1}(\Omega^{\varepsilon})}+\int_{\Omega^{\varepsilon}}\left(-v_{\varepsilon}c_{\varepsilon}^{\pm}+\nabla c_{\varepsilon}^{\pm}\pm\varepsilon^{\gamma}c_{\varepsilon}^{\pm}\nabla\Phi_{\varepsilon}\right)\cdot\nabla\varphi_{3}dx
		\nonumber \\&\qquad \qquad \qquad \qquad \qquad \quad  =\int_{\Omega^{\varepsilon}}R_{\varepsilon}^{\pm}\left(c_{\varepsilon}^{+},c_{\varepsilon}^{-}\right)\varphi_{3}dx.
	\end{align}
	for all $\left(\varphi_{1},\varphi_{2},\varphi_{3},\psi\right)\in\left[H_{0}^{1}\left(\Omega^{\varepsilon}\right)\right]^{d}\times H^{1}\left(\Omega^{\varepsilon}\right)\times H^{1}\left(\Omega^{\varepsilon}\right)\times H^{1}\left(\Omega^{\varepsilon}\right)$.
\end{defn}
\begin{thm}\label{thm:existence}
	\emph{\textbf{Existence and uniqueness of solutions}}
	
	Assume $\left(\mbox{A}_{1}\right)$-$\left(\mbox{A}_{6}\right)$. For
	each $\varepsilon>0$, the microscopic problem $\left(P^{\varepsilon}\right)$ admits a unique
	weak solution $\left(v_{\varepsilon},p_{\varepsilon},\Phi_{\varepsilon},c_{\varepsilon}^{\pm}\right)$
	in the sense of Definition \ref{def:mathbb=00007BPeps=00007D}.
\end{thm}

The proof of Theorem \ref{thm:existence} can be found in \cite{RMK12} (see Theorem 3.7) and \cite{ray2013thesis}.
\begin{thm}
	\emph{\textbf{Effective transport tensors. Cell problems}}\label{thm:auxiliary}
	
	The averaged macroscopic permittivity/diffusion tensor $\mathbb{D}=\left(D_{ij}\right)_{1\le i,j\le d}$
	is defined by
	\[
	D_{ij}:=\int_{Y_{l}}\left(\delta_{ij}+\partial_{y_{i}}\varphi_{j}\left(y\right)\right)dy,
	\]
	where $\varphi_{j}=\varphi_{j}\left(y\right)$ for $1\le j\le d$
	are unique weak solutions in $H^{1}\left(Y_{l}\right)$ of the following family of cell problems
	\begin{equation}
	\begin{cases}
	-\Delta_{y}\varphi_{j}\left(y\right)=0 & \text{ in }\;Y_{l},\\
	\nabla_{y}\varphi_{j}\left(y\right)\cdot\text{n}=-e_{j}\cdot\text{n} & \text{ on }\;\Gamma,\\
	\varphi_{j}\;\text{ periodic in }\;y.
	\end{cases}\label{eq:cellproblem1}
	\end{equation}
	
	Furthermore, the averaged macroscopic permeability tensor $\mathbb{K}=\left(K_{ij}\right)_{1\le i,j\le d}$
	is defined by
	\[
	K_{ij}:=\int_{Y_{l}}w_{j}^{i}dy,
	\]
	where $w_{j}=w_{j}\left(y\right)$ together with $\pi_{j}=\pi_{j}\left(y\right)$
	for $1\le j\le d$ are unique weak solutions, respectively, in $H^{1}\left(Y_{l}\right)$ and $L^{2}\left(Y_{l}\right)$ of the following family of cell problems
	\begin{equation}
	\begin{cases}
	-\Delta_{y}w_{j}+\nabla_{y}\pi_{j}=e_{j} & \text{ in }\;Y_{l},\\
	\nabla_{y}\cdot w_{j}=0 & \text{ in }\;Y_{l},\\
	w_{j}=0 & \text{ in }\;\Gamma,\\
	w_{j},\pi_{j}\;\text{ periodic in }\;y.
	\end{cases}\label{eq:cellproblem2}
	\end{equation}
	
	Also, we define the following cell problem
	\begin{equation}
	\begin{cases}
	-\Delta_{y}\varphi\left(y\right)=1 & \text{ in }\;Y_{l},\\
	\varphi\left(y\right)=0 & \text{ on }\;\Gamma,\\
	\varphi\;\text{ periodic in }\;y,
	\end{cases}\label{eq:cellproblem3}
	\end{equation}
	which admits a unique weak solution in $H^{1}\left(Y_{l}\right)$.
	
	Note that $\delta_{ij}$ denotes the Kronecker symbol and $e_{j}$
	is the $j$th unit vector of $\mathbb{R}^{d}$.
\end{thm}

The proof of Theorem \ref{thm:auxiliary} can be found in \cite{RMK12} (see Definition 4.4) and \cite{ray2013thesis}.

\begin{rem}
	Fundamental results for elliptic equations provide that the problems
	(\ref{eq:cellproblem1}) and (\ref{eq:cellproblem3}) admit a unique
	weak solution in $H^{1}\left(Y_{l}\right)$ (cf. \cite{CP99}). Similarly, the
	solutions $w_{j}^{i}$ and $\pi_{j}$ ($1\le i,j\le d$) of (\ref{eq:cellproblem2})
	are in $H^{1}\left(Y_{l}\right)$ and $L^{2}\left(Y_{l}\right)$,
	respectively. Particularly, for every $s\in\left(-\frac{1}{2},\frac{1}{2}\right)$
	it follows from Theorem 4 and Theorem 7 in \cite{Giu98} that for
	$1\le i,j\le d$,
	\[
	\varphi_{j}^{i}\in H^{1+s}\left(Y_{l}\right)\text{ and }w_{j}^{i}\in H^{1+s}\left(Y_{l}\right),\pi_{j}\in H^{s}\left(Y_{l}\right)
	\]
	are unique weak solution to (\ref{eq:cellproblem1}) and (\ref{eq:cellproblem2}),
	respectively.
	
	The permeability tensor $\mathbb{K}$ is symmetric
	and positive definite (cf. \cite[Proposition 2.2, Chapter 7]{Sanchez80}),
	whilst the same properties of the permittivity tensor $\mathbb{D}$
	are proven in \cite{CP99}.
\end{rem}

\subsection{Neumann condition for the electrostatic potential}

\begin{thm}
	\label{thm:positive+bound}\emph{\textbf{Positivity and Boundedness of solution}}
	
	Assume $\left(\mbox{A}_{1}\right)$-$\left(\mbox{A}_{4}\right)$.
	Let $\left(v_{\varepsilon},p_{\varepsilon},\Phi_{\varepsilon},c_{\varepsilon}^{\pm}\right)$
	be a weak solution of the microscopic problem $\left(P^{\varepsilon}\right)$ with the Neumann condition \eqref{eq:boundN} 
	in the sense of Definition \ref{def:mathbb=00007BPeps=00007D}. Then
	the concentration fields
	$c_{\varepsilon}^{\pm}$ are non-negative and essentially bounded from above uniformly
	in $\varepsilon$.
\end{thm}

The proof of Theorem \ref{thm:positive+bound} can be found in \cite{RMK12} (see Theorems 3.3 and 3.4) and \cite{ray2013thesis}.
\begin{thm}
	\emph{\textbf{\emph{A priori} estimates}}\label{thm:aprioriestimate-N}
	
	Assume $\left(\mbox{A}_{1}\right)$-$\left(\mbox{A}_{6}\right)$.
	The following \emph{a priori} estimates hold:
	
	For the electrostatic potential, we
	have
	\begin{equation}
	\varepsilon^{\alpha}\left\Vert \Phi_{\varepsilon}\right\Vert _{L^{2}\left(0,T;H^{1}\left(\Omega^{\varepsilon}\right)\right)}\le C.
	\label{eq:thm2.4-1}
	\end{equation}
	If $\beta\ge\alpha$, it holds
	\begin{equation}
	\left\Vert v_{\varepsilon}\right\Vert _{L^{2}\left(\left(0,T\right)\times\Omega^{\varepsilon}\right)}+\varepsilon\left\Vert \nabla v_{\varepsilon}\right\Vert _{L^{2}\left(\left(0,T\right)\times\Omega^{\varepsilon}\right)}\le C,
	\label{eq:thm2.4-2}
	\end{equation}
	and additionally, if $\gamma\ge\alpha$, it holds
	\begin{align}
	& \max_{t\in\left[0,T\right]}\left\Vert c_{\varepsilon}^{-}\right\Vert _{L^{2}\left(\Omega^{\varepsilon}\right)}+\max_{t\in\left[0,T\right]}\left\Vert c_{\varepsilon}^{+}\right\Vert _{L^{2}\left(\Omega^{\varepsilon}\right)}+\left\Vert \nabla c_{\varepsilon}^{-}\right\Vert _{L^{2}\left(\left(0,T\right)\times\Omega^{\varepsilon}\right)}+\left\Vert \nabla c_{\varepsilon}^{+}\right\Vert _{L^{2}\left(\left(0,T\right)\times\Omega^{\varepsilon}\right)} \nonumber \\
	& +\left\Vert \partial_{t}c_{\varepsilon}^{-}\right\Vert _{L^{2}\left(0,T;\left(H^{1}\left(\Omega^{\varepsilon}\right)\right)'\right)}+\left\Vert \partial_{t}c_{\varepsilon}^{+}\right\Vert _{L^{2}\left(0,T;\left(H^{1}\left(\Omega^{\varepsilon}\right)\right)'\right)}\le C.
	\label{eq:thm2.4-3}
	\end{align}
\end{thm}

The proof of Theorem \ref{thm:aprioriestimate-N} can be found in \cite{RMK12} (see Theorem 3.5) and \cite{ray2013thesis}.
\begin{thm}
	\emph{\textbf{Homogenization of ($P_{N}^{\varepsilon}$)}}\label{thm:homoNeumann}
	
	Let the \emph{a priori} estimates \eqref{eq:thm2.4-1}-\eqref{eq:thm2.4-3} of Theorem \ref{thm:aprioriestimate-N}
	be valid. Taking $\tilde{\Phi}_{\varepsilon}:=\varepsilon^{\alpha}\Phi_{\varepsilon}$,
	there exist functions $\tilde{\Phi}_{0}\in L^{2}\left(0,T;H^{1}\left(\Omega\right)\right)$
	and $\tilde{\Phi}_{1}\in L^{2}\left(\left(0,T\right)\times\Omega;H_{\#}^{1}\left(Y\right)\right)$
	such that, up to a subsequence, we have
	\begin{align*}
	& \tilde{\Phi}_{\varepsilon}\stackrel{2}{\rightharpoonup}\tilde{\Phi}_{0},\\
	& \nabla\tilde{\Phi}_{\varepsilon}\stackrel{2}{\rightharpoonup}\nabla_{x}\tilde{\Phi}_{0}+\nabla_{y}\tilde{\Phi}_{1}.
	\end{align*}
	
	If $\beta\ge\alpha$, then there exist functions $v_{0}\in L^{2}\left(\left(0,T\right)\times\Omega;H_{\#}^{1}\left(Y\right)\right)$
	and $p_{0}\in L^{2}\left(\left(0,T\right)\times\Omega\times Y\right)$
	such that, up to a subsequence, we have
	\begin{align*}
	& v_{\varepsilon}\stackrel{2}{\rightharpoonup}v_{0},\\
	& \varepsilon\nabla v_{\varepsilon}\stackrel{2}{\rightharpoonup}\nabla_{y}v_{0},\\
	& p_{\varepsilon}\stackrel{2}{\rightharpoonup}p_{0}.
	\end{align*}
	
	Moreover, the convergence for the pressure is strong in $L^{2}\left(\Omega\right)/\mathbb{R}$.
	
	If $\gamma\ge\alpha$, then there exist functions $c_{0}^{\pm}\in L^{2}\left(0,T;H^{1}\left(\Omega\right)\right)$
	and $c_{1}^{\pm}\in L^{2}\left(\left(0,T\right)\times\Omega;H_{\#}^{1}\left(Y\right)\right)$
	such that, up to a subsequence, we have
	\begin{align*}
	& c_{\varepsilon}^{\pm}\to c_{0}^{\pm}\;\text{ strongly in }\;L^{2}\left(\left(0,T\right)\times\Omega\right),\\
	& \nabla c_{\varepsilon}^{\pm}\stackrel{2}{\rightharpoonup}\nabla_{x}c_{0}^{\pm}+\nabla_{y}c_{1}^{\pm}.
	\end{align*}
\end{thm}
\begin{thm}
	\emph{\textbf{Strong formulation of the macroscopic problem in the Neumann case - $(P^0_{N})$}}\label{thm:macroNeumann}
	
	Let $\left(v_{\varepsilon},p_{\varepsilon},\Phi_{\varepsilon},c_{\varepsilon}^{\pm}\right)$
	be a weak solution of $\left(P^{\varepsilon}\right)$ in the sense
	of Definition \ref{def:mathbb=00007BPeps=00007D}. According to Theorem \ref{thm:homoNeumann}, we have the
	following results:
	
	Let $\tilde{\Phi}_{0}$ be the two-scale limit of the electrostatic
		potential $\tilde{\Phi}_{\varepsilon}$, it then satisfies the following
		macroscopic system:
		\[
		\begin{cases}
		-\nabla_{x}\cdot\left(\mathbb{D}\nabla_{x}\tilde{\Phi}_{0}\left(t,x\right)\right)=\bar{\sigma}+\left|Y_{l}\right|\left(c_{0}^{+}\left(t,x\right)-c_{0}^{-}\left(t,x\right)\right) & \text{in }\;\left(0,T\right)\times\Omega,\\
		\mathbb{D}\nabla_{x}\tilde{\Phi}_{0}\left(t,x\right)\cdot\text{n}=0 & \text{on }\;\left(0,T\right)\times\partial\Omega,
		\end{cases}
		\]
		where $\bar{\sigma}:=\int_{\Gamma}\sigma dS_{y}$ and the permittivity/diffusion
		tensor $\mathbb{D}$ is defined in Theorem \ref{thm:auxiliary}.
		
		Let $v_{0}$ be the two-scale limit of the velocity field $v_{\varepsilon}$.
		With additionally $\beta\ge\alpha$, it then satisfies the following
		macroscopic system:
		\[
		\begin{cases}
		\bar{v}_{0}\left(t,x\right)+\mathbb{K}\nabla_{x}p_{0}\left(t,x\right)=-\mathbb{K}\left(c_{0}^{+}-c_{0}^{-}\right)\nabla_{x}\tilde{\Phi}_{0}\left(t,x\right) & \text{in }\;\left(0,T\right)\times\Omega,\;\text{if }\;\beta=\alpha,\\
		\bar{v}_{0}\left(t,x\right)+\mathbb{K}\nabla_{x}p_{0}\left(t,x\right)=0 & \text{in }\;\left(0,T\right)\times\Omega,\;\text{if }\;\beta>\alpha,\\
		\nabla_{x}\cdot\bar{v}_{0}\left(t,x\right)=0 & \text{in }\;\left(0,T\right)\times\Omega,\\
		\bar{v}_{0}\left(t,x\right)\cdot\text{n}=0 & \text{on }\;\left(0,T\right)\times\partial\Omega,
		\end{cases}
		\]
		where $\bar{v}_{0}\left(t,x\right)=\int_{Y_{l}}v_{0}\left(t,x,y\right)dy$
		and the permeability tensor $\mathbb{K}$ is defined in Theorem \ref{thm:auxiliary}.
		
		Let $c_{0}^{\pm}$ be the two-scale limits of the concentration fields
		$c_{\varepsilon}^{\pm}$. With $\gamma=\alpha$, they satisfy the
		following macroscopic system:
		\[
		\begin{cases}
		\left|Y_{l}\right|\partial_{t}c_{0}^{\pm}\left(t,x\right)+\nabla_{x}\cdot\left[c_{0}^{\pm}\left(t,x\right)\left(\bar{v}_{0}\mp\mathbb{D}\nabla_{x}\tilde{\Phi}_{0}\right)-\mathbb{D}\nabla_{x}c_{0}^{\pm}\left(t,x\right)\right]\\
		=\left|Y_{l}\right|R_{0}^{\pm}\left(c_{0}^{+}\left(t,x\right),c_{0}^{-}\left(t,x\right)\right) & \text{in }\;\left(0,T\right)\times\Omega,\\
		\left(c_{0}^{\pm}\left(t,x\right)\left(\bar{v}_{0}\left(t,x\right)\mp\mathbb{D}\nabla_{x}\tilde{\Phi}_{0}\left(t,x\right)\right)-\mathbb{D}\nabla_{x}c_{0}^{\pm}\left(t,x\right)\right)\cdot\text{n}=0 & \text{on }\;\left(0,T\right)\times\partial\Omega,
		\end{cases}
		\]
		while with $\gamma>\alpha$, they satisfy
		\[
		\begin{cases}
		\left|Y_{l}\right|\partial_{t}c_{0}^{\pm}\left(t,x\right)+\nabla_{x}\cdot\left[c_{0}^{\pm}\left(t,x\right)\bar{v}_{0}\left(t,x\right)-\mathbb{D}\nabla_{x}c_{0}^{\pm}\left(t,x\right)\right]\\
		=\left|Y_{l}\right|R_{0}^{\pm}\left(c_{0}^{+}\left(t,x\right),c_{0}^{-}\left(t,x\right)\right) & \text{in }\;\left(0,T\right)\times\Omega,\\
		\left(c_{0}^{\pm}\left(t,x\right)\bar{v}_{0}\left(t,x\right)-\mathbb{D}\nabla_{x}c_{0}^{\pm}\left(t,x\right)\right)\cdot\text{n}=0 & \text{on }\;\left(0,T\right)\times\partial\Omega.
		\end{cases}
		\]
\end{thm}

\begin{rem}
	Due to the \emph{a priori} estimate \eqref{eq:thm2.4-1} for the electrostatic potential in Theorem \ref{thm:aprioriestimate-N}, $\Phi_{\varepsilon}$
	and its gradient $\nabla\Phi_{\varepsilon}$ converge to zero when
	$\alpha<0$. In Theorem \ref{thm:macroNeumann}, the number densities
	$c_{0}^{\pm}$ in the macroscopic Poisson equations with permittivity
	tensor $\mathbb{D}$ positions itself as forcing terms. Similarly,
	the forcing terms in the macroscopic Stokes equations with the case
	$\beta=\alpha$ dwell in the part of the electrostatic potential $\tilde{\Phi}_{0}$
	and the distribution of the number densities $c_{0}^{\pm}$. Clearly,
	the macroscopic Nernst-Planck equations in the case $\gamma=\alpha$
	yield the fully coupled system of partial differential equations,
	whilst with $\gamma>\alpha$ it reduces to a convection-diffusion-reaction
	system due to also the structure of the reaction terms $R_{0}^{\pm}$.
\end{rem}

Let us define the function space
\[
H_{N}^{1}\left(\Omega\right):=\left\{ v\in H^{1}\left(\Omega\right):-\mathbb{D}\nabla_{x}u\cdot\text{n}=0\;\text{ on }\partial\Omega\right\},
\]
which is a closed subspace of $H^1(\Omega)$. This Hilbert space plays a role when writing the weak formulation of the macroscopic systems in Theorem \ref{thm:macroNeumann-1} and Theorem \ref{thm:macroDirichlet-1}.

\begin{thm}
	\emph{\textbf{Weak formulation of $(P^0_N)$}}\label{thm:macroNeumann-1}
	
	Let the quadruple of functions $\left(v_{0},p_{0},\tilde{\Phi}_{0},c_{0}^{\pm}\right)$
	be defined as in Theorem \ref{thm:macroNeumann}. Then, it satisfies
	\begin{align*}
	& \bar{v}_{0}\in L^{2}\left(\left(0,T\right)\times\Omega\right),p_{0}\in L^{2}\left(\left(0,T\right)\times\Omega\right),\\
	& \tilde{\Phi}_{0}\in L^{2}\left(0,T;H^{1}\left(\Omega\right)\right),c_{0}^{\pm}\in L^{2}\left(0,T;H^{1}\left(\Omega\right)\right),\partial_{t}c_{0}^{\pm}\in L^{2}\left(0,T;\left(H^{1}\left(\Omega\right)\right)'\right)
	\end{align*}
	and becomes a weak solution to $\left(P_{N}^{0}\right)$ provided
	that
	\begin{align*}
	& \int_{\Omega}\left(\bar{v}_{0}\varphi_{1}-\mathbb{K}p_{0}\nabla\cdot\varphi_{1}\right)dx=-\mathbb{K}\int_{\Omega}\left(c_{0}^{+}-c_{0}^{-}\right)\nabla\tilde{\Phi}_{0}\cdot\varphi_{1}dx\;\text{ if }\;\beta=\alpha,\\
	& \int_{\Omega}\left(\bar{v}_{0}\varphi_{1}-\mathbb{K}p_{0}\nabla\cdot\varphi_{1}\right)dx=0\;\text{ if }\;\beta>\alpha,\\
	& \int_{\Omega}\bar{v}_{0}\cdot\nabla\psi dx=0,\\
	& \int_{\Omega}\left|Y_{l}\right|^{-1}\mathbb{D}\nabla\tilde{\Phi}_{0}\cdot\nabla\varphi_{2}dx-\left|Y_{l}\right|^{-1}\bar{\sigma}\int_{\Omega}\varphi_{2}dx=\int_{\Omega}\left(c_{0}^{+}-c_{0}^{-}\right)\varphi_{2}dx,\\
	& \left\langle \partial_{t}c_{0}^{\pm},\varphi_{3}\right\rangle _{\left(H^{1}\right)',H^{1}}+\int_{\Omega}\left|Y_{l}\right|^{-1}\left(-c_{0}^{\pm}\left(\bar{v}_{0}\mp\mathbb{D}\nabla\tilde{\Phi}_{0}\right)+\mathbb{D}\nabla c_{0}^{\pm}\right)\cdot\nabla\varphi_{3}dx
	\\ &\qquad \qquad \qquad \qquad
	=\int_{\Omega}R_{0}^{\pm}\left(c_{0}^{+},c_{0}^{-}\right)\varphi_{3}dx\;\text{ if }\;\gamma=\alpha,\\
	& \left\langle \partial_{t}c_{0}^{\pm},\varphi_{3}\right\rangle _{\left(H^{1}\right)',H^{1}}+\int_{\Omega}\left|Y_{l}\right|^{-1}\left(-c_{0}^{\pm}\bar{v}_{0}+\mathbb{D}\nabla c_{0}^{\pm}\right)\cdot\nabla\varphi_{3}dx
	\\ & \qquad \qquad \qquad \qquad
	=\int_{\Omega}R_{0}^{\pm}\left(c_{0}^{+},c_{0}^{-}\right)\varphi_{3}dx\;\text{ if }\;\gamma>\alpha,
	\end{align*}
	for all $\left(\varphi_{1},\varphi_{2},\varphi_{3},\psi\right)\in\left[H_{0}^{1}\left(\Omega\right)\right]^{d}\times H^{1}_{N}\left(\Omega\right)\times H^{1}\left(\Omega\right)\times H^{1}\left(\Omega\right)$.
\end{thm}

The proof of Theorems \ref{thm:homoNeumann}, \ref{thm:macroNeumann} and \ref{thm:macroNeumann-1} are collected from Theorems 4.5--4.10 in \cite{RMK12} and can also be found in \cite{ray2013thesis}.
\subsection{Dirichlet condition for the electrostatic potential}
\begin{remark}
	In Theorem \ref{thm:positive+bound}, the proof (as mentioned in \cite[Theorem 3.3, Theorem 3.4]{RMK12}) consists in suitable choices of test functions, based on the energy-estimates arguments. Nevertheless,  for the problem where the Dirichlet boundary condition \eqref{eq:boundD} is prescribed, the volume additivity constraint $c_{\varepsilon}^{+}+c_{\varepsilon}^{-}=1$ is required to guarantee the $\varepsilon$-independent boundedness of the concentration fields.
\end{remark}
\begin{defn}
	\label{def:Phihom}Assume $\left(\mbox{A}_{1}\right)$-$\left(\mbox{A}_{4}\right)$.
	Let $\Phi_{\varepsilon}$ be a solution of the microscopic problem
	$\left(P^{\varepsilon}\right)$ in the sense of Definition \ref{def:mathbb=00007BPeps=00007D}.
	Then the transformed electrostatic potential $\Phi_{\varepsilon}^{\text{hom}}:=\Phi_{\varepsilon}-\Phi_{D}$
	satifies the following system:
	\begin{align*}
		& -\varepsilon^{\alpha}\Delta\Phi_{\varepsilon}^{\text{hom}}=c_{\varepsilon}^{+}-c_{\varepsilon}^{-}\quad\text{in }\;Q_{T}^{\varepsilon},\\
		& \Phi_{\varepsilon}^{\text{hom}}=0\quad\text{in }\;\left(0,T\right)\times\Gamma_{D}^{\varepsilon},\\
		& \varepsilon^{\alpha}\nabla\Phi_{\varepsilon}^{\text{hom}}\cdot\text{n}=0\quad\text{in }\;\left(0,T\right)\times\partial\Omega.
	\end{align*}
\end{defn}

\begin{thm}\emph{\textbf{\emph{A priori} estimates}}\label{thm:aprioriestimate-D}
	
	Assume $\left(\mbox{A}_{1}\right)$-$\left(\mbox{A}_{4}\right)$. The following \emph{a priori} estimates hold:
	
	For the electrostatic potential,
	we have
	\begin{equation}
	\varepsilon^{\alpha-2}\left\Vert \Phi_{\varepsilon}^{\text{hom}}\right\Vert _{L^{2}\left(\left(0,T\right)\times\Omega^{\varepsilon}\right)}+\varepsilon^{\alpha-1}\left\Vert \nabla\Phi_{\varepsilon}^{\text{hom}}\right\Vert _{L^{2}\left(\left(0,T\right)\times\Omega^{\varepsilon}\right)}\le C.
	\label{eq:thm2.4-4}
	\end{equation}
	If $\beta\ge\alpha-1$, it holds
	\begin{equation}
	\left\Vert v_{\varepsilon}\right\Vert _{L^{2}\left(\left(0,T\right)\times\Omega^{\varepsilon}\right)}+\varepsilon\left\Vert \nabla v_{\varepsilon}\right\Vert _{L^{2}\left(\left(0,T\right)\times\Omega^{\varepsilon}\right)}\le C,
	\label{eq:thm2.4-5}
	\end{equation}
	and additionally if $\gamma\ge\alpha-1$, it holds
	\begin{align}
	& \max_{t\in\left[0,T\right]}\left\Vert c_{\varepsilon}^{-}\right\Vert _{L^{2}\left(\Omega^{\varepsilon}\right)}+\max_{t\in\left[0,T\right]}\left\Vert c_{\varepsilon}^{+}\right\Vert _{L^{2}\left(\Omega^{\varepsilon}\right)}+\left\Vert \nabla c_{\varepsilon}^{-}\right\Vert _{L^{2}\left(\left(0,T\right)\times\Omega^{\varepsilon}\right)}+\left\Vert \nabla c_{\varepsilon}^{+}\right\Vert _{L^{2}\left(\left(0,T\right)\times\Omega^{\varepsilon}\right)} \nonumber \\
	& +\left\Vert \partial_{t}c_{\varepsilon}^{-}\right\Vert _{L^{2}\left(0,T;\left(H^{1}\left(\Omega^{\varepsilon}\right)\right)'\right)}+\left\Vert \partial_{t}c_{\varepsilon}^{+}\right\Vert _{L^{2}\left(0,T;\left(H^{1}\left(\Omega^{\varepsilon}\right)\right)'\right)}\le C.
	\label{eq:thm2.4-6}
	\end{align}
\end{thm}

The proof of Theorem \ref{thm:aprioriestimate-D} can be found in \cite{RMK12} (see Theorem 3.6) and \cite{ray2013thesis}.
\begin{thm}
	\emph{\textbf{Homogenization of ($P_{D}^{\varepsilon}$)}}\label{thm:homoDirichlet}
	
	Let the \emph{a priori} estimates \eqref{eq:thm2.4-4}-\eqref{eq:thm2.4-6} of Theorem \ref{thm:aprioriestimate-D}
	be valid. Let $\Phi_{\varepsilon}^{\text{hom}}$
	be as defined in Definition \ref{def:Phihom}. Taking $\tilde{\Phi}_{\varepsilon}:=\varepsilon^{\alpha-2}\Phi_{\varepsilon}^{\text{hom}}$,
	then it satisfies the following system:
	\begin{align*}
		& -\varepsilon^{2}\Delta\tilde{\Phi}_{\varepsilon}=c_{\varepsilon}^{+}-c_{\varepsilon}^{-}\;\text{ in }\;Q_{T}^{\varepsilon},\\
		& \tilde{\Phi}_{\varepsilon}=0\;\text{ in }\;\left(0,T\right)\times\Gamma_{\varepsilon},\\
		& \varepsilon^{2}\nabla\tilde{\Phi}_{\varepsilon}\cdot\text{n}=0\;\text{ in }\;\left(0,T\right)\times\partial\Omega.
	\end{align*}
	
	Therefore, we can find a function $\tilde{\Phi}_{0}\in L^{2}\left(\left(0,T\right)\times\Omega;H_{\#}^{1}\left(Y\right)\right)$
	such that, up to a subsequence,
	\begin{align*}
		& \tilde{\Phi}_{\varepsilon}\stackrel{2}{\rightharpoonup}\tilde{\Phi}_{0},\\
		& \varepsilon\nabla\tilde{\Phi}_{\varepsilon}\stackrel{2}{\rightharpoonup}\nabla_{y}\tilde{\Phi}_{0}.
	\end{align*}
	
	If additionally $\beta\ge\alpha-1$, then there exist functions $v_{0}\in L^{2}\left(\left(0,T\right)\times\Omega;H_{\#}^{1}\left(Y\right)\right)$
	and $p_{0}\left(t,x,y\right)\in L^{2}\left(\left(0,T\right)\times\Omega\times Y\right)$
	such that, up to a subsequence, we have
	\begin{align*}
		& v_{\varepsilon}\stackrel{2}{\rightharpoonup}v_{0},\\
		& \varepsilon\nabla v_{\varepsilon}\stackrel{2}{\rightharpoonup}\nabla_{y}v_{0},\\
		& p_{\varepsilon}\stackrel{2}{\rightharpoonup}p_{0}.
	\end{align*}
	
	Furthermore, there exist functions $c_{0}^{\pm}\in L^{2}\left(0,T;H^{1}\left(\Omega\right)\right)$
	and $c_{1}^{\pm}\in L^{2}\left(\left(0,T\right)\times\Omega;H_{\#}^{1}\left(Y\right)\right)$
	such that, up to a subsequence, we have
	\begin{align*}
		& c_{\varepsilon}^{\pm}\to c_{0}^{\pm}\;\text{ strongly in }\;L^{2}\left(\left(0,T\right)\times\Omega\right),\\
		& \nabla c_{\varepsilon}^{\pm}\stackrel{2}{\rightharpoonup}\nabla_{x}c_{0}^{\pm}+\nabla_{y}c_{1}^{\pm}.
	\end{align*}
\end{thm}
\begin{thm}
	\emph{\textbf{Strong formulation of the macroscopic problem in the Dirichlet case - $(P^0_{D})$}}\label{thm:macroDirichlet}
	
	Let $\left(v_{\varepsilon},p_{\varepsilon},\Phi_{\varepsilon},c_{\varepsilon}^{\pm}\right)$
	be a weak solution of $\left(P^{\varepsilon}\right)$ in the sense
	of Definition \ref{def:mathbb=00007BPeps=00007D}. According to Theorem \ref{thm:homoDirichlet}, we have
	the following results:
	
	Let $\tilde{\Phi}_{0}$ be the two-scale limit of the electrostatic
		potential $\tilde{\Phi}_{\varepsilon}$, it then satisfies the following
		macroscopic equation:
		\[
		\overline{\tilde{\Phi}}_{0}\left(t,x\right)=\left(\int_{Y_{l}}\varphi\left(y\right)dy\right)\left(c_{0}^{+}\left(t,x\right)-c_{0}^{-}\left(t,x\right)\right),
		\]
		where $\overline{\tilde{\Phi}}_{0}\left(t,x\right)=\int_{Y_{l}}\tilde{\Phi}_{0}\left(t,x,y\right)dy$
		and $\varphi$ is the solution of the cell problem (\ref{eq:cellproblem3}).
		
		Let $v_{0}$ be the two-scale limit of the velocity field $v_{\varepsilon}$.
		With $\beta\ge\alpha-1$, it then satisfies the following macroscopic
		system:
		\[
		\begin{cases}
		\bar{v}_{0}\left(t,x\right)+\mathbb{K}\nabla_{x}p_{0}\left(t,x\right)=0 & \text{in}\;\left(0,T\right)\times\Omega,\\
		\nabla_{x}\cdot\bar{v}_{0}\left(t,x\right)=0 & \text{in}\;\left(0,T\right)\times\Omega,\\
		\bar{v}_{0}\left(t,x\right)\cdot\text{n}=0 & \text{on}\;\left(0,T\right)\times\partial\Omega,
		\end{cases}
		\]
		where $\bar{v}_{0}\left(t,x\right)=\int_{Y_{l}}v_{0}\left(t,x,y\right)dy$
		and the permeability tensor $\mathbb{K}$ is defined in Theorem \ref{thm:auxiliary}.
		
		Let $c_{0}^{\pm}$ be the two-scale limits of the concentration fields
		$c_{\varepsilon}^{\pm}$. With $\gamma\ge\alpha-1$, they satisfy
		the following macroscopic system:
		\[
		\begin{cases}
		\left|Y_{l}\right|\partial_{t}c_{0}^{\pm}\left(t,x\right)+\nabla_{x}\cdot\left[c_{0}^{\pm}\left(t,x\right)\bar{v}_{0}\left(t,x\right)-\mathbb{D}\nabla_{x}c_{0}^{\pm}\left(t,x\right)\right]\\
		=\left|Y_{l}\right|R_{0}^{\pm}\left(c_{0}^{+}\left(t,x\right),c_{0}^{-}\left(t,x\right)\right) & \text{in}\;\left(0,T\right)\times\Omega,\\
		\left(c_{0}^{\pm}\left(t,x\right)\bar{v}_{0}\left(t,x\right)-\mathbb{D}\nabla_{x}c_{0}^{\pm}\left(t,x\right)\right)\cdot\text{n}=0 & \text{on}\;\left(0,T\right)\times\partial\Omega.
		\end{cases}
		\]
		where the permittivity/diffusion tensor $\mathbb{D}$ is defined in
		Theorem \ref{thm:auxiliary}.
\end{thm}
\begin{rem}
	Due to the \emph{a priori} estimate for the electrostatic potential
	in Theorem \ref{thm:aprioriestimate-D}, $\Phi_{\varepsilon}$ converges
	to $\Phi_{D}$ as $\alpha<2$. Moreover, in the case $\alpha<1$ we
	obtain the convergence of $\Phi_{\varepsilon}$ and its gradient $\nabla\Phi_{\varepsilon}$
	to the $\zeta$-potential $\Phi_{D}$ and zero, respectively. When
	$\alpha=2$, then $\tilde{\Phi}_{\varepsilon}=\Phi_{\varepsilon}^{\text{hom}}:=\Phi_{\varepsilon}-\Phi_{D}$
	holds, we compute that
	\begin{equation}
		\bar{\Phi}_{0}\left(t,x\right)=\int_{Y_{l}}\left(\Phi_{0}^{\text{hom}}\left(t,x,y\right)+\Phi_{D}\right)dy=\left(\int_{Y_{l}}\varphi\left(y\right)dy\right)\left(c_{0}^{+}\left(t,x\right)-c_{0}^{-}\left(t,x\right)\right)+\left|Y_{l}\right|\Phi_{D}.\label{eq:specialpotential}
	\end{equation}
	
	In Theorem \ref{thm:macroDirichlet}, we see that in contrast to Theorem
	\ref{thm:macroNeumann}, the electrostatic potential is not present
	in the macroscopic Stokes and Nernst-Planck equations. In addition,
	the macroscopic Poisson system for the electrostatic potential reduces
	from the partial differential equations in the Neumann case to the
	macroscopic ``representation'' in the Dirichlet case. Both cases
	are all coupled with the concentration fields $c_{0}^{\pm}$. Note
	that in both Neumann and Dirichlet cases, we need the strong convergence
	of the concentration fields, i.e. $c_{\varepsilon}^{\pm}\to c_{0}^{\pm}$
	in $L^{2}\left(\left(0,T\right)\times\Omega\right)$, to derive the
	macroscopic systems for the electrostatic potential, the fluid flow
	as well as for the pressure, respectively.
\end{rem}
\begin{thm}
	\emph{\textbf{Weak formulation of $(P^0_{D})$}}\label{thm:macroDirichlet-1}
	
	Let the quadruple of functions $\left(v_{0},p_{0},\tilde{\Phi}_{0},c_{0}^{\pm}\right)$
	be defined as in Theorem \ref{thm:macroDirichlet}. Then, it satisfies
	\begin{align*}
		& \bar{v}_{0}\in L^{2}\left(\left(0,T\right)\times\Omega\right),p_{0}\in L^{2}\left(\left(0,T\right)\times\Omega\right),\\
		& \tilde{\Phi}_{0}\in L^{2}\left(\left(0,T\right)\times\Omega\right),c_{0}^{\pm}\in L^{2}\left(0,T;H^{1}\left(\Omega\right)\right),\;\partial_{t}c_{0}^{\pm}\in L^{2}\left(0,T;\left(H^{1}\left(\Omega\right)\right)'\right)
	\end{align*}
	and is a weak solution to $\left(P_{D}^{0}\right)$ provided
	that
	\begin{align*}
		& \int_{\Omega}\left(\bar{v}_{0}\varphi_{1}-\mathbb{K}p_{0}\nabla\cdot\varphi_{1}\right)dx=0,\\
		& \int_{\Omega}\bar{v}_{0}\cdot\nabla\psi dx=0,\\
		& \int_{\Omega}\overline{\tilde{\Phi}}_{0}\varphi_{2}dx=\left(\int_{Y_{l}}\varphi\left(y\right)dy\right)\int_{\Omega}\left(c_{0}^{+}-c_{0}^{-}\right)\varphi_{2}dx,\\
		& \left\langle \partial_{t}c_{0}^{\pm},\varphi_{3}\right\rangle _{\left(H^{1}\right)',H^{1}}+\int_{\Omega}\left|Y_{l}\right|^{-1}\left(-c_{0}^{\pm}\bar{v}_{0}+\mathbb{D}\nabla c_{0}^{\pm}\right)\cdot\nabla\varphi_{3}dx=\int_{\Omega}R_{0}^{\pm}\left(c_{0}^{+},c_{0}^{-}\right)\varphi_{3}dx,
	\end{align*}
	for all $\left(\varphi_{1},\varphi_{2},\varphi_{3},\psi\right)\in\left[H_{0}^{1}\left(\Omega\right)\right]^{d}\times H^{1}_{N}\left(\Omega\right)\times H^{1}\left(\Omega\right)\times H^{1}\left(\Omega\right)$.
\end{thm}

The proof of Theorems \ref{thm:homoDirichlet}, \ref{thm:macroDirichlet} and \ref{thm:macroDirichlet-1} are collected from Theorems 4.11--4.16 in \cite{RMK12} and can also be found in \cite{ray2013thesis}.

\subsection{Discussions}
According to proofs of the macroscopic systems in Theorems 4.6, 4.8, 4.10, 4.12, 4.14 and 4.16 cf.  \cite{ray2013thesis}, we formulate here the first-order
limit functions of the systems ($P_{N}^{0}$) and ($P_{D}^{0}$), respectively.

When the electric potential satisfies the Neumann condition on the micro-surface, we deduce that $\tilde{\Phi}_{1}$ can be 
	formulated by
	\[
	\tilde{\Phi}_{1}\left(t,x,y\right)=\sum_{j=1}^{d}\varphi_{j}\left(y\right)\partial_{x_{j}}\tilde{\Phi}_{0}\left(t,x\right),
	\]
	with $\varphi_{j}$ being solutions of the cell problems (\ref{eq:cellproblem1}).
	We also remark that the limit function $p_{0}$ for the pressure is
	proved to be independent of $y$, i.e. $p_{0}\left(t,x,y\right)=p_{0}\left(t,x\right)$,
	due to the structure of the Stokes equation, see Theorem \ref{thm:homoNeumann}.
	Accordingly, the representation of the limit function $v_{0}$ for
	the fluid flow is given by
	\[
	v_{0}\left(t,x,y\right)=\begin{cases}
	-{\displaystyle \sum_{j=1}^{d}}w_{j}\left(y\right)\left[\left(c_{0}^{+}-c_{0}^{-}\right)\partial_{x_{j}}\tilde{\Phi}_{0}\left(t,x\right)+\partial_{x_{j}}p_{0}\left(t,x\right)\right] & \text{if }\;\beta=\alpha,\\
	-{\displaystyle \sum_{j=1}^{d}}w_{j}\left(y\right)\partial_{x_{j}}p_{0}\left(t,x\right) & \text{if }\;\beta>\alpha,
	\end{cases}
	\]
	where $w_{j}=w_{j}\left(y\right)$ for $1\le j\le d$ are the solutions
	of the cell problems (\ref{eq:cellproblem2}). We are able
	to determine the (extended) macroscopic Darcy's law by the following
	pressure:
	\[
	\tilde{p}_{1}\left(t,x,y\right)=p_{1}\left(t,x,y\right)+\left(c_{0}^{+}\left(t,x\right)-c_{0}^{-}\left(t,x\right)\right)\tilde{\Phi}_{1}\left(t,x,y\right),
	\]
	where with $\pi_{j}=\pi_{j}\left(y\right)$ for $1\le j\le d$ are
	the solutions of the cell problems (\ref{eq:cellproblem2}), we compute
	that
	\[
	p_{1}\left(t,x,y\right)=\begin{cases}
	-{\displaystyle \sum_{j=1}^{d}}\pi_{j}\left(y\right)\left[\left(c_{0}^{+}-c_{0}^{-}\right)\partial_{x_{j}}\tilde{\Phi}_{0}\left(t,x\right)+\partial_{x_{j}}p_{0}\left(t,x\right)\right] & \text{if }\;\beta=\alpha,\\
	-{\displaystyle \sum_{j=1}^{d}}\pi_{j}\left(y\right)\partial_{x_{j}}p_{0}\left(t,x\right) & \text{if }\;\beta>\alpha.
	\end{cases}
	\]
	
	On the other hand, the representation of the first-order functions
	$c_{1}^{\pm}$ is
	\[
	c_{1}^{\pm}\left(t,x,y\right)=\begin{cases}
	{\displaystyle \sum_{j=1}^{d}}\left(\varphi_{j}\left(y\right)\partial_{x_{j}}c_{0}^{\pm}\left(t,x\right)\mp c_{0}^{\pm}\left(t,x\right)\partial_{x_{j}}\tilde{\Phi}_{0}\left(t,x\right)\right) & \text{if }\;\gamma=\alpha,\\
	{\displaystyle \sum_{j=1}^{d}}\varphi_{j}\left(y\right)\partial_{x_{j}}c_{0}^{\pm}\left(t,x\right) & \text{if }\;\gamma>\alpha,
	\end{cases}
	\]
	where $\varphi_{j}=\varphi_{j}\left(y\right)$ for $1\le j\le d$
	are the solutions of the cell problems (\ref{eq:cellproblem1}).

When the electric potential satisfies the Dirichlet condition on the micro-surface, we obtain a different scenario. In fact,
	the macroscopic electrostatic potential $\tilde{\Phi}_{0}$ is in this
	case dependent of $y$ and it can be computed by the averaged 
	term $\overline{\tilde{\Phi}}_{0}$ (see Theorem \ref{thm:macroDirichlet}
	and the special case in (\ref{eq:specialpotential})). We obtain the
	same manner with the macroscopic velocity $v_{0}$ in Theorem \ref{thm:macroDirichlet}.
	However, the limit function $p_{0}$ for the pressure remains independent
	of $y$. As a consequence, the representation of the first-order
	functions $c_{1}^{\pm}$ is
	\[
	c_{1}^{\pm}\left(t,x,y\right)=\begin{cases}
	{\displaystyle \sum_{j=1}^{d}}\left(\varphi_{j}\left(y\right)\partial_{x_{j}}c_{0}^{\pm}\left(t,x\right)\mp c_{0}^{\pm}\left(t,x\right)\tilde{\Phi}_{0}\left(t,x\right)\right) & \text{if }\;\gamma=\alpha-1,\\
	{\displaystyle \sum_{j=1}^{d}}\varphi_{j}\left(y\right)\partial_{x_{j}}c_{0}^{\pm}\left(t,x\right) & \text{if }\;\gamma>\alpha-1,
	\end{cases}
	\]
	where $\varphi_{j}=\varphi_{j}\left(y\right)$ for $1\le j\le d$
	are the solutions of the cell problems (\ref{eq:cellproblem1}).

It is worth mentioning that upscaling the microscopic system $\left(P^{\varepsilon}\right)$ is done by the two-scale convergence method. This approach, which aims to derive the limit system, does not require the derivation of the first-order macroscopic velocity, denoted by $v_{1}$ herein. To gain the corrector for the oscillating pressure arising in the Stokes equation, we use the same procedures as in \cite{PM96}, and thus, we need the structure of $v_{1}$.

Following \cite{Sanchez80}, we have in the Neumann case for the electrostatic potential that
\[
v_{1}\left(t,x,\frac{x}{\varepsilon}\right)=\begin{cases}
-{\displaystyle \sum_{i,j=1}^{d}}r_{ij}\left(\dfrac{x}{\varepsilon}\right)\partial_{x_{i}}\left(\left(c_{0}^{+}-c_{0}^{-}\right)\partial_{x_{j}}\tilde{\Phi}_{0}\left(t,x\right)+\partial_{x_{j}}p_{0}\left(t,x\right)\right) & \text{if }\;\beta=\alpha,\\
-{\displaystyle \sum_{i,j=1}^{d}}r_{ij}\left(\dfrac{x}{\varepsilon}\right)\partial_{x_{i}x_{j}}^{2}p_{0}\left(t,x\right) & \text{if }\;\beta>\alpha,
\end{cases}
\]
where $r_{ij}\in H^{1}\left(Y_{l}\right)$ for $1\le i,j\le d$ is
the solution for the following cell problem
\begin{equation}
\begin{cases}
\nabla_{y}\cdot r_{ij}+w_{j}^{i}=\left|Y_{l}\right|^{-1}K_{ij} & \text{in }\;Y_{l},\\
r_{ij}=0 & \text{on }\;\Gamma,\\
r_{ij}\;\text{ periodic in }\;y.
\end{cases}\label{eq:lastcellproblem}
\end{equation}

It holds
\[
v_{1}\left(t,x,\frac{x}{\varepsilon}\right) = -{\displaystyle \sum_{i,j=1}^{d}}r_{ij}\left(\dfrac{x}{\varepsilon}\right)\partial_{x_{i}x_{j}}^{2}p_{0}\left(t,x\right),
\]
provided the electrostatic potential satisfies the Dirichlet boundary data on the micro-surfaces.

\subsection{Auxiliary estimates}
Here, let $Y_{l}$ and $\Omega^{\varepsilon}$
as defined in Subsubsection \ref{subsec:A-geometrical-interpretation}.
\begin{lem}
	\label{lem:ines1}(cf. \cite{KM17}) Let $p^{\varepsilon}\left(x\right):=p\left(x/\varepsilon\right)\in H^{1}\left(\Omega^{\varepsilon}\right)$
	satisfy
	\[
	\bar{p}:=\frac{1}{\left|Y_{l}\right|}\int_{Y_{l}}p\left(y\right)dy,
	\]
	then the following estimate holds:
	\[
	\left\Vert p^{\varepsilon}-\bar{p}\right\Vert _{L^{2}\left(\Omega^{\varepsilon}\right)}\le C\varepsilon^{\frac{1}{2}}\left\Vert p^{\varepsilon}\right\Vert _{H^{1}\left(\Omega^{\varepsilon}\right)}.
	\]
\end{lem}
\begin{lem}
	\label{lem:ines2} Assume $\partial\Omega\in C^{k}$
	for $k\ge 4$ holds. Then, there exist $\delta_{0}>0$ and
	a function $\eta^{\delta}\in\left[C^{k-1}\left(\overline{\Omega}\right)\right]^{d}$
	such that $\eta^{\delta}=\bar{v}_{0}$ on $\partial\Omega$ with $\bar{v}_{0}$
	being the averaged macroscopic velocity defined in Theorem \ref{thm:macroNeumann},
	$\nabla_{x}\cdot\eta^{\delta}=0$ in $\Omega$ and for any $1\le q\le\infty$
	and $0\le\ell\le k-1$, the following estimate holds:
	\begin{align}\label{bdttrum-6}
	\left\Vert \nabla^{\ell}\eta^{\delta}\right\Vert _{L^{q}\left(\Omega\right)}\le C\delta^{\frac{1}{q}-\ell}\quad\text{for }\delta\in\left(0,\delta_{0}\right].
	\end{align}
\end{lem}
\begin{proof}
	We adapt the notation from \cite{PM96} (see Lemma 1) to our proof here. It is well known from \cite[Lemma 14.16]{GT83} that there exists an $\varepsilon$-independent
	$\gamma>0$ such that the distance function $z\left(x\right)=\text{dist}\left(x,\partial\Omega\right)$
	belongs to $C^{k}\left(\mathcal{S}_{\gamma}\right)$ where \begin{align}
	 \mathcal{S}_{\gamma}:=\left\{ x\in\overline{\Omega}:\text{dist}\left(x,\partial\Omega\right)\le\gamma\right\}.
	\end{align}
	By definition, we have
	\[
	\partial\Omega:=\left\{ x\in\mathbb{R}^{d}:z\left(x\right)=0\right\} \;\text{and }\text{n}:=-\frac{\nabla z}{\left|\nabla z\right|}\quad\text{for }x\in\mathcal{S}_{\gamma}.
	\]
	
	If we define a function $V\left(z,\xi\right)$ by
	\begin{equation}
	V\left(z,\xi\right):=-\frac{\bar{v}_{0}\left(x\right)}{\left|\nabla z\left(x\right)\right|}\quad\text{for }x=x\left(z,\xi\right)\in\mathcal{S}_{\gamma}\label{eq:functionV}
	\end{equation}
	where $\xi$ is the tangential component of $z$ along $\partial\Omega$. 
	We observe that $\left|\nabla z\right|>0$ for $x\in\mathcal{S}_{\gamma}$
	and the trace $V\left(0,\xi\right)$ is well-defined as a function
	in $C^{k}\left(\mathcal{S}_{\gamma}\right)$.
	
	Following the same spirit of the argument as in Temam \cite{Temam79} in e.g. Proposition 2.3, we aim to take $\eta^{\delta}$
	as $\text{curl}\psi$, where $\psi$ is chosen in such a way that
	\[
	\frac{\partial\psi}{\partial\tau}=0\quad\text{on }\partial\Omega,
	\]
	where we denote by $\tau$ the tangential component of $\psi$, and
	\[
	\nabla\psi\cdot\text{n}=\bar{v}_{0}\cdot\tau\quad\text{on }\partial\Omega.
	\]
	
	Note from the structure of the macroscopic Stokes system (cf. Theorem \ref{thm:macroNeumann} and Theorem \ref{thm:macroDirichlet}) that $\bar{v}_{0}\cdot\text{n}=0$
	on $\partial\Omega$ and from the fact that the tangential component
	is different from 0 in principle. We aim to choose $\psi=0$ on $\partial\Omega$.
	Based on the function $V\left(z,\xi\right)$, defined in (\ref{eq:functionV}),
	we choose
	\[
	\psi\left(x\right)=z\left(x\right)\text{exp}\left(-\frac{z\left(x\right)}{\delta}\right)V\left(0,\xi\right)\cdot\tau\left(x\right).
	\]
	
	Due to the presence of $z$, it is clear that $\psi=0$ on $\partial\Omega$.
	Furthermore, we can check that
	\[
	\nabla\psi\cdot\text{n}=-\frac{\nabla z}{\left|\nabla z\right|}\cdot\left(\nabla z\frac{\partial\psi}{\partial z}\right)=-\left|\nabla z\right|\left(1-\frac{z}{\delta}\right)\text{exp}\left(-\frac{z}{\delta}\right)V\left(0,\xi\right)\cdot\tau\left(x\right)=\bar{v}_{0}\cdot\tau
	\]
	holds on $\partial\Omega$.
	
	Therefore, we are now allowed to take $\eta^{\delta}=\text{curl}\psi$
	in $\mathcal{S}_{\gamma}$.
	
	We can now complete the proof of the lemma. Indeed,
	we estimate that
	\begin{align*}
	\left\Vert \nabla\psi\right\Vert _{L^{q}\left(\mathcal{S}_{\gamma}\right)}^{q} & \le C\int_{\mathcal{S}_{\gamma}}\left(\left|\left(1-\frac{z}{\delta}\right)\text{exp}\left(-\frac{z}{\delta}\right)V\left(0,\xi\right)\right|^{2}+\left|z\text{exp}\left(-\frac{z}{\delta}\right)\frac{\partial V}{\partial\xi}\left(0,\xi\right)\right|^{2}\right)^{\frac{q}{2}}dx\\
	& \le C\delta.
	\end{align*}
	
	Owning to the $C^{k}$-smoothness of $\partial\Omega$, we can proceed
	as above to obtain the following high-order estimate: 
	\[
	\left\Vert \nabla^{\ell+1}\psi\right\Vert _{L^{q}\left(\mathcal{S}_{\gamma}\right)}\le C\delta^{\frac{1}{q}-\ell}\quad\text{for }0\le\ell\le k-1.
	\]
	
	Hence, for $\delta\ll\gamma$ the function $\psi$ is exponentially
	small at $\bar{\mathcal{S}}_{\gamma}=\left\{ x\in\overline{\Omega}:\text{dist}\left(x,\partial\Omega\right)=\gamma\right\} $
	and we can extend it to a function, which is denoted again by $\psi$,
	in $C^{k}\left(\overline{\Omega}\right)$ such that it satisfies $\eta^{\delta}=\text{curl}\psi$
	and thus the estimate \eqref{bdttrum-6}.
\end{proof}

By Lemma \ref{lem:ines2}, we can introduce a cut-off function $m^{\varepsilon}\in\mathcal{D}\left(\overline{\Omega}\right)$
corresponding to $\partial\Omega$, satisfying
\[
m^{\varepsilon}(x)=\begin{cases}
0 & \text{if }\text{dist}\left(x,\partial\Omega\right)\le\varepsilon,\\
1 & \text{if }\text{dist}\left(x,\partial\Omega\right)\ge2\varepsilon,
\end{cases}\quad\text{and}\quad\left\Vert \nabla^{\ell}m^{\varepsilon}\right\Vert _{L^{\infty}\left(\Omega\right)}\le C\varepsilon^{-\ell}\text{ for }\ell\in\left[0,2\right].
\]

As a consequence, one can also show that
\begin{equation}
	\left\Vert 1-m^{\varepsilon}\right\Vert _{L^{2}\left(\Omega^{\varepsilon}\right)}\le C\varepsilon^{\frac{1}{2}},\quad\varepsilon\left\Vert \nabla m^{\varepsilon}\right\Vert _{L^{2}\left(\Omega^{\varepsilon}\right)}\le C\varepsilon^{\frac{1}{2}}.\label{eq:cutin}
\end{equation}

\begin{lem}
	\label{lem:ines3}(cf. \cite[Lemma 1, Appendix]{Sanchez80}) For any $u\in H_{0}^{1}(\Omega^{\varepsilon})$, it holds
	\[
	\left\Vert u\right\Vert _{L^{2}\left(\Omega^{\varepsilon}\right)}\le C\varepsilon\left\Vert \nabla u\right\Vert _{\left[L^{2}\left(\Omega^{\varepsilon}\right)\right]^{d}}.
	\]
\end{lem}

\section{Macroscopic reconstructions and corrector estimates}\label{mainsec}

In this section, we begin by introducing the so-called macroscopic
reconstructions and provide supplementary estimates needed for the
proof of our main results stated in Theorem \ref{thm:main1-6} and Theorem
\ref{thm:main2-6}. Our working methodology was used in \cite{Eck2} and
successfully applied to derive the corrector estimates for a thermo-diffusion
system in a uniformly periodic medium (cf. \cite{KM17}) and an advection-diffusion-reaction
system in a locally-periodic medium (cf. \cite{Tycho}). In principle,
the asymptotic expansion can be justified by estimating the differences
of the solutions of the microscopic model $\left(P^{\varepsilon}\right)$
and macroscopic reconstructions which can be defined from the macroscopic
models $\left(P_{N}^{0}\right)$ and $\left(P_{D}^{0}\right)$.

Our main results correspond to two cases:
\begin{tcolorbox}
\begin{description}
	\item [{Case\;1:}] The electric potential satisfies the Neumann boundary condition at the boundary of the perforations
	\item [{Case\;2:}] The electric potential satisfies the Dirichlet boundary condition at the boundary of the perforations
\end{description}
\end{tcolorbox}

\begin{rem}
	To gain the corrector estimates, we require more regularity assumptions on the involved functions as well as the smoothness of the boundaries of the macroscopic domain; compare to the assumptions obtained when upscaling $\left(P^{\varepsilon}\right)$. In fact, it is worth pointing out that in Theorem \ref{thm:main1-6} and Theorem \ref{thm:main2-6} we assume the regularity properties on the limit functions, postulated in Theorem \ref{thm:macroNeumann-1} for Case 1 and in Theorem \ref{thm:macroDirichlet-1} for Case 2, as follows:
	\begin{equation}\label{dk1-6}
	\tilde{\Phi}_{0},c_{0}^{\pm}\in W^{1,\infty}\left(\Omega^{\varepsilon}\right)\cap H^{2}\left(\Omega^{\varepsilon}\right),\bar{v}_{0}\in L^{\infty}\left(\Omega^{\varepsilon}\right).
	\end{equation}
	The cell functions $\varphi_{j}$ for $1\le j\le d$ solving
	the family of cell problems (\ref{eq:cellproblem1}) are supposed to fulfill 
	\begin{equation}\label{dk2-6}
	\varphi_{j}\in W^{1+s,2}\left(Y_{l}\right) \text{ for } s>d/2.
	\end{equation}
	Moreover, the cell functions $w_{j}^{i}$, $\pi_{j}$ and $r_{ij}$ for $1\le i,j\le d$ solving the cell problems (\ref{eq:cellproblem2})
	and (\ref{eq:lastcellproblem}), respectively, satisfy \begin{equation}\label{dk3-6}
	w_{j}^{i}\in W^{2+s,2}\left(Y_{l}\right),
	\pi_{j}\in W^{1+s,2}\left(Y_{l}\right) \text{ and } r_{ij}\in W^{1+s,2}\left(Y_{l}\right)
	\text{ for } s>d/2.
	\end{equation}
	In addition, we stress that the corrector estimates for the Stokes equation can be gained if we take  $\partial\Omega\in C^{4}$. This assumption is only needed to handle Lemma \ref{lem:ines2}.
	 
\end{rem}

\subsection{Main results}
\begin{thm}
	\label{thm:main1-6}\emph{\textbf{Corrector estimates for Case 1}}
	
	Assume $\left(\text{A}_{1}\right)-\left(\text{A}_{6}\right)$.
	Let the quadruples $\left(v_{\varepsilon},p_{\varepsilon},\Phi_{\varepsilon},c_{\varepsilon}^{\pm}\right)$
	and $\left(v_{0},p_{0},\Phi_{0},c_{0}^{\pm}\right)$ be weak solutions
	to $\left(P^{\varepsilon}\right)$ and $\left(P_{N}^{0}\right)$ in
	the sense of Definition \ref{def:mathbb=00007BPeps=00007D} and Theorem
	\ref{thm:macroNeumann-1}, respectively. Furthermore, we assume that
	the limit solutions satisfy the regularity property \eqref{dk1-6}.
	Let $\varphi_{j}$ for $1\le j\le d$ be the cell functions solving
	the family of cell problems (\ref{eq:cellproblem1}) and satisfy \eqref{dk2-6}. Assume that the initial homogenization limit is of the rate
	\[
	\left\Vert c_{\varepsilon}^{\pm,0}-c_{0}^{\pm,0}\right\Vert _{L^{2}\left(\Omega^{\varepsilon}\right)}^{2}\le C\varepsilon^{\mu}\quad\text{for some }\mu\in\mathbb{R}_{+}.
	\]
	Then the following corrector estimates hold:
	\begin{align*}
	& \left\Vert v_{\varepsilon}-\bar{v}_{0}^{\varepsilon}\right\Vert _{L^{2}\left(\left(0,T\right)\times\Omega^{\varepsilon}\right)}\le C\varepsilon^{\frac{1}{2}},\\
	& \left\Vert \tilde{\Phi}_{\varepsilon}-\tilde{\Phi}_{0}^{\varepsilon}\right\Vert _{L^{2}\left(\left(0,T\right)\times\Omega^{\varepsilon}\right)}+\left\Vert c_{\varepsilon}^{\pm}-c_{0}^{\pm,\varepsilon}\right\Vert _{L^{2}\left(\left(0,T\right)\times\Omega^{\varepsilon}\right)}\\
	&+\left\Vert \nabla\left(\tilde{\Phi}_{\varepsilon}-\tilde{\Phi}_{1}^{\varepsilon}\right)\right\Vert _{\left[L^{2}\left(\left(0,T\right)\times\Omega^{\varepsilon}\right)\right]^{d}}\le C\max\left\{ \varepsilon^{\frac{1}{2}},\varepsilon^{\frac{\mu}{2}}\right\} ,\\
	& \left\Vert \nabla\left(c_{\varepsilon}^{\pm}-c_{1}^{\pm,\varepsilon}\right)\right\Vert _{\left[L^{2}\left(\left(0,T\right)\times\Omega^{\varepsilon}\right)\right]^{d}}\le C\max\left\{ \varepsilon^{\frac{1}{4}},\varepsilon^{\frac{\mu}{2}}\right\} ,
	\end{align*}
	where $\bar{v}_{0}^{\varepsilon}$, $\Phi_{0}^{\varepsilon}$, $c_{0}^{\pm,\varepsilon}$,
	$\tilde{\Phi}_{1}^{\varepsilon}$, $c_{1}^{\pm,\varepsilon}$ are
	the macroscopic reconstructions defined in (\ref{eq:macrorecon1})-(\ref{eq:macrorecon5}).
	
	Let $w_{j}^{i}$, $\pi_{j}$ and $r_{ij}$ for $1\le i,j\le d$ be
	the cell functions solving the cell problems (\ref{eq:cellproblem2})
	and (\ref{eq:lastcellproblem}), respectively, and satisfy \eqref{dk3-6}. If we further assume that
	\[
	\tilde{\Phi}_{0}\in H^{4}\left(\Omega^{\varepsilon}\right),c_{0}^{\pm}\in W^{2,\infty}\left(\Omega^{\varepsilon}\right),p_{0}\in H^{4}\left(\Omega^{\varepsilon}\right),
	\]
	then for any $\lambda\in\left(0,1\right)$, the following
	corrector estimates hold:
	\begin{align*}
	& \left\Vert v_{\varepsilon}-\left|Y_{l}\right|^{-1}\mathbb{D}v_{0}^{\varepsilon}-\varepsilon\left|Y_{l}\right|^{-1}\mathbb{D}v_{1}^{\varepsilon}\right\Vert _{\left[L^{2}\left(\left(0,T\right)\times\Omega^{\varepsilon}\right)\right]^{d}}\le C\left(\max\left\{ \varepsilon^{\frac{1}{2}},\varepsilon^{\frac{\mu}{2}}\right\} +\varepsilon^{\frac{\lambda}{2}}+\varepsilon^{1-\frac{3\lambda}{2}}+\varepsilon^{\frac{1}{2}-\lambda}\right),\\
	& \left\Vert p_{\varepsilon}-p_{0}\right\Vert _{L^{2}\left(\Omega\right)/\mathbb{R}}\le C\left(\max\left\{ \varepsilon^{\frac{1}{2}},\varepsilon^{\frac{\mu}{2}}\right\} +\varepsilon^{\frac{\lambda}{2}}+\varepsilon^{1-\frac{3\lambda}{2}}+\varepsilon^{\frac{1}{2}-\lambda}\right),
	\end{align*}
	where $v_{0}^{\varepsilon}$ and $v_{1}^{\varepsilon}$ are defined
	in (\ref{eq:macrorecon6}) and (\ref{eq:macrorecon7}), respectively.
\end{thm}

\begin{thm}
	\label{thm:main2-6}\emph{\textbf{Corrector estimates for Case 2}}
	
	Assume $\left(\text{A}_{1}\right)-\left(\text{A}_{4}\right)$. Let the quadruples $\left(v_{\varepsilon},p_{\varepsilon},\Phi_{\varepsilon},c_{\varepsilon}^{\pm}\right)$
	and $\left(v_{0},p_{0},\Phi_{0},c_{0}^{\pm}\right)$ be weak solutions
	to $\left(P^{\varepsilon}\right)$ and $\left(P_{D}^{0}\right)$ in
	the sense of Definition \ref{def:mathbb=00007BPeps=00007D} and Theorem
	\ref{thm:macroDirichlet-1}, respectively. Furthermore, we assume
	that the limit solutions satisfy the regularity property \eqref{dk1-6}.
	Let $\varphi_{j}$ for $1\le j\le d$ be the cell functions solving
	the family of cell problems (\ref{eq:cellproblem1}) and satisfy \eqref{dk2-6}. Assume that the initial homogenization limit is of the rate
	\[
	\left\Vert c_{\varepsilon}^{\pm,0}-c_{0}^{\pm,0}\right\Vert _{L^{2}\left(\Omega^{\varepsilon}\right)}^{2}\le C\varepsilon^{\mu}\quad\text{for some }\mu\in\mathbb{R}_{+}.
	\]
	Then the following corrector estimates hold:
	\begin{align*} & \left\Vert \tilde{\Phi}_{\varepsilon}-\tilde{\Phi}_{0}^{\varepsilon}\right\Vert _{L^{2}\left(\left(0,T\right)\times\Omega^{\varepsilon}\right)}+\left\Vert \nabla\left(\tilde{\Phi}_{\varepsilon}-\tilde{\Phi}_{0}^{\varepsilon}\right)\right\Vert _{\left[L^{2}\left(\left(0,T\right)\times\Omega^{\varepsilon}\right)\right]^{d}}+\left\Vert c_{\varepsilon}^{\pm}-c_{0}^{\pm,\varepsilon}\right\Vert _{L^{2}\left(\left(0,T\right)\times\Omega^{\varepsilon}\right)}\\
	& +\left\Vert \nabla\left(c_{\varepsilon}^{\pm}-c_{1}^{\pm,\varepsilon}\right)\right\Vert _{\left[L^{2}\left(\left(0,T\right)\times\Omega^{\varepsilon}\right)\right]^{d}}+\left\Vert \tilde{\Phi}_{\varepsilon}-\overline{\tilde{\Phi}}_{0}^{\varepsilon}\right\Vert _{L^{2}\left(\left(0,T\right)\times\Omega^{\varepsilon}\right)}\le C\max\left\{ \varepsilon^{\frac{1}{2}},\varepsilon^{\frac{\mu}{2}}\right\} ,
	\end{align*}
	where $c_{0}^{\pm,\varepsilon}$, $c_{1}^{\pm,\varepsilon}$, $\Phi_{0}^{\varepsilon}$,
	$\overline{\tilde{\Phi}}_{0}^{\varepsilon}$ are the macroscopic reconstructions
	defined in (\ref{eq:3.51-6-1})-(\ref{eq:3.52-6-1}) and (\ref{eq:3.53-6-1})-(\ref{eq:3.54-6-1}).
	
	Let $w_{j}^{i}$, $\pi_{j}$ and $r_{ij}$ $1\le i,j\le d$ be the cell functions solving
	the cell problems (\ref{eq:cellproblem2}) and (\ref{eq:lastcellproblem}),
	respectively, and satisfy \eqref{dk3-6}. If we further assume that $p_{0}\in H^{4}\left(\Omega^{\varepsilon}\right)$, then for any $\lambda\in\left(0,1\right)$, the following
	corrector estimates hold:
	\begin{align*}
	& \left\Vert v_{\varepsilon}-\left|Y_{l}\right|^{-1}\mathbb{D}v_{0}^{\varepsilon}-\varepsilon\left|Y_{l}\right|^{-1}\mathbb{D}v_{1}^{\varepsilon}\right\Vert _{\left[L^{2}\left(\left(0,T\right)\times\Omega^{\varepsilon}\right)\right]^{d}}\le C\left(\max\left\{ \varepsilon^{\frac{1}{2}},\varepsilon^{\frac{\mu}{2}}\right\} +\varepsilon^{\frac{\lambda}{2}}+\varepsilon^{1-\frac{3\lambda}{2}}+\varepsilon^{\frac{1}{2}-\lambda}\right),\\
	& \left\Vert p_{\varepsilon}-p_{0}\right\Vert _{L^{2}\left(\Omega\right)/\mathbb{R}}\le C\left(\max\left\{ \varepsilon^{\frac{1}{2}},\varepsilon^{\frac{\mu}{2}}\right\} +\varepsilon^{\frac{\lambda}{2}}+\varepsilon^{1-\frac{3\lambda}{2}}+\varepsilon^{\frac{1}{2}-\lambda}\right),
	\end{align*}
	where $v_{0}^{\varepsilon}$ and $v_{1}^{\varepsilon}$ are defined
	in (\ref{eq:3.49-6-1}) and (\ref{eq:3.50-6-1}), respectively.
\end{thm}

\subsection{Proof of Theorem \ref{thm:main1-6}}

To study the homogenization limit, the existence of asymptotic
expansions
\begin{align*}
	v_{\varepsilon}\left(t,x\right) & =v_{0}\left(t,x,\frac{x}{\varepsilon}\right)+\varepsilon v_{1}\left(t,x,\frac{x}{\varepsilon}\right)+\varepsilon^{2}v_{2}\left(t,x,\frac{x}{\varepsilon}\right)+...\\
	p_{\varepsilon}\left(t,x\right) & =p_{0}\left(t,x,\frac{x}{\varepsilon}\right)+\varepsilon p_{1}\left(t,x,\frac{x}{\varepsilon}\right)+\varepsilon^{2}p_{2}\left(t,x,\frac{x}{\varepsilon}\right)+...\\
	\tilde{\Phi}_{\varepsilon}\left(t,x\right) & =\tilde{\Phi}_{0}\left(t,x,\frac{x}{\varepsilon}\right)+\varepsilon\tilde{\Phi}_{1}\left(t,x,\frac{x}{\varepsilon}\right)+\varepsilon^{2}\tilde{\Phi}_{2}\left(t,x,\frac{x}{\varepsilon}\right)+...\\
	c_{\varepsilon}^{\pm}\left(t,x\right) & =c_{0}^{\pm}\left(t,x,\frac{x}{\varepsilon}\right)+\varepsilon c_{1}^{\pm}\left(t,x,\frac{x}{\varepsilon}\right)+\varepsilon^{2}c_{2}^{\pm}\left(t,x,\frac{x}{\varepsilon}\right)+...,
\end{align*}
is assumed and some terms (e.g. $v_{0},p_{0},\tilde{\Phi}_{0},c_{0}^{\pm}$)
have been determined in the previous section. Since the route to derive
the corrector for Stokes' equation is different from the usual construction
of corrector estimates for the other equations, we shall postpone for a moment the proof of the corrector for the pressure.

We define the macroscopic reconstructions, as
follows:
\begin{align}
	& \bar{v}_{0}^{\varepsilon}\left(t,x\right):=\left|Y_{l}\right|^{-1}\bar{v}_{0}\left(t,x\right),\label{eq:macrorecon1}\\
	& \tilde{\Phi}_{0}^{\varepsilon}\left(t,x\right):=\tilde{\Phi}_{0}\left(t,x\right),\\
	& \tilde{\Phi}_{1}^{\varepsilon}\left(t,x\right):=\tilde{\Phi}_{0}^{\varepsilon}\left(t,x\right)+\varepsilon\sum_{j=1}^{d}\varphi_{j}\left(\frac{x}{\varepsilon}\right)\partial_{x_{j}}\tilde{\Phi}_{0}^{\varepsilon}\left(t,x\right),\\
	& c_{0}^{\pm,\varepsilon}\left(t,x\right):=c_{0}^{\pm}\left(t,x\right),\\
	& c_{1}^{\pm,\varepsilon}\left(t,x\right):=c_{0}^{\pm,\varepsilon}\left(t,x\right)+\varepsilon\sum_{j=1}^{d}\varphi_{j}\left(\frac{x}{\varepsilon}\right)\partial_{x_{j}}c_{0}^{\pm,\varepsilon}\left(t,x\right),\label{eq:macrorecon5}\\
	& v_{0}^{\varepsilon}\left(t,x\right):=v_{0}\left(t,x,\frac{x}{\varepsilon}\right),\label{eq:macrorecon6}\\
	& v_{1}^{\varepsilon}\left(t,x\right):=v_{1}\left(t,x,\frac{x}{\varepsilon}\right).\label{eq:macrorecon7}
\end{align}
Lemma \ref{lem:ines1} ensures the following
estimate:
\begin{equation}
	\left\Vert v_{\varepsilon}-\bar{v}_{0}^{\varepsilon}\right\Vert _{L^{2}\left(\left(0,T\right)\times\Omega^{\varepsilon}\right)}\le C\varepsilon^{\frac{1}{2}},\label{eq:estimate1-6}
\end{equation}
where Definition \ref{def:mathbb=00007BPeps=00007D} and Theorem \ref{thm:existence} guarantee the
regularity for $v_{\varepsilon}$.

Let us now consider the correctors for the electrostatic potential
and the concentrations. We take the difference of the microscopic and macroscopic
Poisson equations in Definition \ref{def:mathbb=00007BPeps=00007D}
and Theorem \ref{thm:macroNeumann}, respectively, with the test function
$\varphi_{2}\in H^{1}\left(\Omega^{\varepsilon}\right)$ and thus
obtain
\begin{align}
	\int_{\Omega^{\varepsilon}}\left(\nabla\tilde{\Phi}_{\varepsilon}-\left|Y_{l}\right|^{-1}\mathbb{D}\nabla\tilde{\Phi}_{0}\right)\cdot\nabla\varphi_{2}dx & +\left|Y_{l}\right|^{-1}\bar{\sigma}\int_{\Omega^{\varepsilon}}\varphi_{2}dx-\varepsilon\int_{\Gamma^{\varepsilon}}\sigma\varphi_{2}dS_{\varepsilon}\nonumber \\
	& =\int_{\Omega^{\varepsilon}}\left(c_{\varepsilon}^{+}-c_{0}^{+}+c_{0}^{-}-c_{\varepsilon}^{-}\right)\varphi_{2}dx,\label{eq:3.2-6-1}
\end{align}
where we recall that $\tilde{\Phi}_{\varepsilon}=\varepsilon^{\alpha}\Phi_{\varepsilon}$
cf. Theorem \ref{thm:homoNeumann}.

Similarly, for $\varphi_{3}\in H^{1}\left(\Omega^{\varepsilon}\right)$
we also find the difference equations for the Nernst-Planck equations,
as follows:
\begin{align}
	& \left\langle \partial_{t}\left(c_{\varepsilon}^{\pm}-c_{0}^{\pm}\right),\varphi_{3}\right\rangle _{\left(H^{1}\right)',H^{1}}+\int_{\Omega^{\varepsilon}}\left(\nabla c_{\varepsilon}^{\pm}-\left|Y_{l}\right|^{-1}\mathbb{D}\nabla c_{0}^{\pm}\right)\cdot\nabla\varphi_{3}dx\nonumber \\
	&\qquad \qquad \qquad \qquad \qquad \quad +\int_{\Omega^{\varepsilon}}\left[\left|Y_{l}\right|^{-1}c_{0}^{\pm}\left(\bar{v}_{0}\mp\mathbb{D}\nabla\tilde{\Phi}_{0}\right)-c_{\varepsilon}^{\pm}\left(v_{\varepsilon}\mp\nabla\tilde{\Phi}_{\varepsilon}\right)\right]\cdot\nabla\varphi_{3}dx
	\nonumber\\ & \qquad \qquad \qquad \qquad \qquad \quad
	=\int_{\Omega^{\varepsilon}}\left(R_{\varepsilon}^{\pm}\left(c_{\varepsilon}^{+},c_{\varepsilon}^{-}\right)-R_{0}^{\pm}\left(c_{0}^{+},c_{0}^{-}\right)\right)\varphi_{3}dx.\label{eq:3.3-6-1}
\end{align}

We start the investigation of these corrector justifications
by the following choice of test functions:
\begin{align}
	\varphi_{2}\left(t,x\right) & :=\tilde{\Phi}_{\varepsilon}\left(t,x\right)-\left(\tilde{\Phi}_{0}^{\varepsilon}\left(t,x\right)+\varepsilon m^{\varepsilon}\left(x\right)\sum_{j=1}^{d}\varphi_{j}\left(\frac{x}{\varepsilon}\right)\partial_{x_{j}}\tilde{\Phi}_{0}\left(t,x\right)\right),\label{eq:choice1}\\
	\varphi_{3}\left(t,x\right) & :=c_{\varepsilon}^{\pm}\left(t,x\right)-\left(c_{0}^{\pm,\varepsilon}\left(t,x\right)+\varepsilon m^{\varepsilon}\left(x\right)\sum_{j=1}^{d}\varphi_{j}\left(\frac{x}{\varepsilon}\right)\partial_{x_{j}}c_{0}^{\pm}\left(t,x\right)\right).\label{eq:choice2}
\end{align}

To get the estimates from (\ref{eq:3.2-6-1}) and (\ref{eq:3.3-6-1}),
we denote the following terms just for ease of presentation:
\begin{align*}
	& \mathcal{J}_{1}:=\int_{\Omega^{\varepsilon}}\left(\nabla\tilde{\Phi}_{\varepsilon}-\left|Y_{l}\right|^{-1}\mathbb{D}\nabla\tilde{\Phi}_{0}\right)\cdot\nabla\varphi_{2}dx,\\
	& \mathcal{J}_{2}:=\left|Y_{l}\right|^{-1}\bar{\sigma}\int_{\Omega^{\varepsilon}}\varphi_{2}dx-\varepsilon\int_{\Gamma^{\varepsilon}}\sigma\varphi_{2}dS_{\varepsilon},\\
	& \mathcal{J}_{3}:=\int_{\Omega^{\varepsilon}}\left(c_{\varepsilon}^{+}-c_{0}^{+}+c_{0}^{-}-c_{\varepsilon}^{-}\right)\varphi_{2}dx,\\
	& \mathcal{K}_{1}:=\left\langle \partial_{t}\left(c_{\varepsilon}^{\pm}-c_{0}^{\pm}\right),\varphi_{3}\right\rangle _{\left(H^{1}\right)',H^{1}}=\int_{\Omega^{\varepsilon}}\partial_{t}\left(c_{\varepsilon}^{\pm}-c_{0}^{\pm}\right)\varphi_{3}dx,\\
	& \mathcal{K}_{2}:=\int_{\Omega^{\varepsilon}}\left(\nabla c_{\varepsilon}^{\pm}-\left|Y_{l}\right|^{-1}\mathbb{D}\nabla c_{0}^{\pm}\right)\cdot\nabla\varphi_{3}dx,\\
	& \mathcal{K}_{3}:=\int_{\Omega^{\varepsilon}}\left[\left|Y_{l}\right|^{-1}c_{0}^{\pm}\left(\bar{v}_{0}\mp\mathbb{D}\nabla\tilde{\Phi}_{0}\right)-c_{\varepsilon}^{\pm}\left(v_{\varepsilon}\mp\nabla\tilde{\Phi}_{\varepsilon}\right)\right]\cdot\nabla\varphi_{3}dx,\\
	& \mathcal{K}_{4}:=\int_{\Omega^{\varepsilon}}\left(R_{\varepsilon}^{\pm}\left(c_{\varepsilon}^{+},c_{\varepsilon}^{-}\right)-R_{0}^{\pm}\left(c_{0}^{+},c_{0}^{-}\right)\right)\varphi_{3}dx.
\end{align*}
Using the representation
\[
\nabla\tilde{\Phi}_{\varepsilon}-\left|Y_{l}\right|^{-1}\mathbb{D}\nabla\tilde{\Phi}_{0}=\nabla\left(\tilde{\Phi}_{\varepsilon}-\tilde{\Phi}_{1}^{\varepsilon}\right)+\nabla\tilde{\Phi}_{1}^{\varepsilon}-\left|Y_{l}\right|^{-1}\mathbb{D}\nabla\tilde{\Phi}_{0},
\]
the term $\mathcal{J}_{1}$ thus becomes
\[
\mathcal{J}_{1}=\int_{\Omega^{\varepsilon}}\nabla\left(\tilde{\Phi}_{\varepsilon}-\tilde{\Phi}_{1}^{\varepsilon}\right)\cdot\nabla\varphi_{2}dx+\int_{\Omega^{\varepsilon}}\left(\nabla\tilde{\Phi}_{1}^{\varepsilon}-\left|Y_{l}\right|^{-1}\mathbb{D}\nabla\tilde{\Phi}_{0}\right)\cdot\nabla\varphi_{2}dx.
\]
With the choice of $\varphi_{2}$ in (\ref{eq:choice1}),
we have
\begin{align}
	\int_{\Omega^{\varepsilon}}\nabla\left(\tilde{\Phi}_{\varepsilon}-\tilde{\Phi}_{1}^{\varepsilon}\right)\cdot\nabla\varphi_{2}dx
	&\ge C\left\Vert \nabla\left(\tilde{\Phi}_{\varepsilon}-\tilde{\Phi}_{1}^{\varepsilon}\right)\right\Vert _{\left[L^{2}\left(\Omega^{\varepsilon}\right)\right]^{d}}^{2}
	\nonumber\\ &
	-C\varepsilon^{2}\left\Vert \nabla\left(\left(1-m^{\varepsilon}\right)\sum_{j=1}^{d}\varphi_{j}^{\varepsilon}\partial_{x_{j}}\tilde{\Phi}_{0}\right)\right\Vert _{\left[L^{2}\left(\Omega^{\varepsilon}\right)\right]^{d}}^{2}.\label{eq:3.6-6-1}
\end{align}

To estimate the second term on the right-hand side of (\ref{eq:3.6-6-1}),
we assume that $\tilde{\Phi}_{0}\in W^{1,\infty}\left(\Omega^{\varepsilon}\right)\cap H^{2}\left(\Omega^{\varepsilon}\right)$
and $\varphi_{j}\in W^{1+s,2}\left(Y_{l}\right)$ for $s>d/2$ and
$1\le j\le d$. Using the Sobolev embedding $W^{1+s,2}\left(Y_{l}\right)\subset C^{1}\left(\bar{Y}_{l}\right)$
together with the inequalities in (\ref{eq:cutin}), we estimate that
\begin{align*}
	\varepsilon\left\Vert \nabla\left(\left(1-m^{\varepsilon}\right)\sum_{j=1}^{d}\varphi_{j}^{\varepsilon}\partial_{x_{j}}\tilde{\Phi}_{0}\right)\right\Vert _{\left[L^{2}\left(\Omega^{\varepsilon}\right)\right]^{d}} & \le\varepsilon\left\Vert \nabla m^{\varepsilon}\right\Vert _{\left[L^{2}\left(\Omega^{\varepsilon}\right)\right]^{d}}\left\Vert \tilde{\Phi}_{0}\right\Vert _{W^{1,\infty}\left(\Omega^{\varepsilon}\right)}\sum_{j=1}^{d}\left\Vert \varphi_{j}\right\Vert _{C\left(\bar{Y}_{l}\right)}\\
	& +\left\Vert 1-m^{\varepsilon}\right\Vert _{L^{2}\left(\Omega^{\varepsilon}\right)}\left\Vert \tilde{\Phi}_{0}\right\Vert _{W^{1,\infty}\left(\Omega^{\varepsilon}\right)}\sum_{j=1}^{d}\left\Vert \nabla_{y}\varphi_{j}\right\Vert _{\left[C\left(\bar{Y}_{l}\right)\right]^{d}}\\
	& +\varepsilon\left\Vert \tilde{\Phi}_{0}\right\Vert _{H^{2}\left(\Omega^{\varepsilon}\right)}\sum_{j=1}^{d}\left\Vert \varphi_{j}\right\Vert _{C\left(\bar{Y}_{l}\right)}\\
	& \le C\left(\varepsilon+\varepsilon^{\frac{1}{2}}\right).
\end{align*}

Taking into account the explicit computation of $\nabla\tilde{\Phi}_{1}^{\varepsilon}$
, which reads
\[
\nabla\tilde{\Phi}_{1}^{\varepsilon}=\nabla_{x}\tilde{\Phi}_{0}+\left(\nabla_{y}\bar{\varphi}\right)^{\varepsilon}\nabla_{x}\tilde{\Phi}_{0}+\varepsilon\bar{\varphi}^{\varepsilon}\nabla_{x}\nabla\tilde{\Phi}_{0}\quad\text{for }\bar{\varphi}=\left(\varphi_{j}\right)_{j=\overline{1,d}},
\]
we can write
\begin{equation}
	\nabla\tilde{\Phi}_{1}^{\varepsilon}-\left|Y_{l}\right|^{-1}\mathbb{D}\nabla\tilde{\Phi}_{0}=\nabla\tilde{\Phi}_{0}+\left(\nabla_{y}\bar{\varphi}\right)^{\varepsilon}\nabla_{x}\tilde{\Phi}_{0}-\left|Y_{l}\right|^{-1}\mathbb{D}\nabla\tilde{\Phi}_{0}+\varepsilon\bar{\varphi}^{\varepsilon}\nabla_{x}\nabla\tilde{\Phi}_{0}.\label{eq:3.7-6-1}
\end{equation}
Due to the smoothness of the involved functions, the fourth
term in (\ref{eq:3.7-6-1}) is bounded in $L^{2}$-norm by
\begin{equation}
	\varepsilon\left\Vert \bar{\varphi}^{\varepsilon}\nabla_{x}\nabla\tilde{\Phi}_{0}\right\Vert _{\left[L^{2}\left(\Omega^{\varepsilon}\right)\right]^{d}}\le C\varepsilon\left\Vert \bar{\varphi}\right\Vert _{\left[C\left(\bar{Y}_{l}\right)\right]^{d}}\left\Vert \tilde{\Phi}_{0}\right\Vert _{H^{2}\left(\Omega^{\varepsilon}\right)}.\label{eq:3.8-6-1}
\end{equation}
On the other hand, from the structure of the
cell problem \ref{eq:cellproblem1} we see that $\mathcal{G}:=\mathbb{I}+\nabla_{y}\bar{\varphi}-\left|Y_{l}\right|^{-1}\mathbb{D}$
is divergence-free with respect to $y$. In parallel with that, its average
also vanishes in the sense that
\[
\int_{Y_{l}}\mathcal{G}dy=0.
\]
Consequently, the function $\mathcal{G}$ possesses a vector potential
$\mathbf{V}$ which is skew-symmetric and satisfies $\mathcal{G}=\nabla_{y}\mathbf{V}$.
Note that the choice of this potential is not unique in general, but
$\mathbf{V}$ can be chosen in such a way that it solves a Poisson
equation $\Delta_{y}\mathbf{V}=f\left(y\right)\nabla_{y}\mathcal{G}$
for some constant $f$ only dependent of the cell's dimension. Therefore,
to determine $\mathbf{V}$ uniquely, we associate this Poisson equation
with the periodic boundary condition at $\Gamma$ and the vanishing
cell average. Using the simple relation $\nabla_{y}=\varepsilon\nabla-\varepsilon\nabla_{x}$,
we arrive at
\begin{equation}
	\mathcal{G}^{\varepsilon}\nabla\tilde{\Phi}_{0}=\varepsilon\nabla\cdot\left(\mathbf{V}^{\varepsilon}\nabla\tilde{\Phi}_{0}\right)-\varepsilon\mathbf{V}^{\varepsilon}\Delta\tilde{\Phi}_{0}.\label{eq:3.8-6}
\end{equation}
Due to the skew-symmetry of $\mathbf{V}$, the first term on the right-hand
side of (\ref{eq:3.8-6}) is divergence-free and its boundedness in
$L^{2}\left(\Omega^{\varepsilon}\right)$ is thus of the order of
$\varepsilon$. Since $\bar{\varphi}\in\left[W^{1+s,2}\left(Y_{l}\right)\right]^{d}$
for $s>d/2$, it yields from the Poisson equation for $\mathbf{V}$
that
\[
\left\Vert \mathbf{V}\right\Vert _{W^{1+s,2}\left(Y_{l}\right)}\le C\left\Vert \mathcal{G}\right\Vert _{W^{s,2}\left(Y_{l}\right)}.
\]

Applying again the compact embedding $W^{s,2}\left(Y_{l}\right)\subset C\left(\bar{Y}_{l}\right)$
for $s>d/2$, we obtain $\mathbf{V}\in C\left(\bar{Y}_{l}\right)$
and it enables us to get the boundedness of the second term on the
right-hand side of (\ref{eq:3.8-6}). In fact, it gives
\[
\varepsilon\left\Vert \mathbf{V}^{\varepsilon}\Delta\tilde{\Phi}_{0}\right\Vert _{L^{2}\left(\Omega^{\varepsilon}\right)}\le\varepsilon\left\Vert \mathbf{V}\right\Vert _{C\left(\bar{Y}_{l}\right)}\left\Vert \tilde{\Phi}_{0}\right\Vert _{H^{2}\left(\Omega^{\varepsilon}\right)}.
\]

Combining this inequality with (\ref{eq:3.7-6-1}), (\ref{eq:3.8-6-1}) and using
the H\"older's inequality, we conclude that
\[
\int_{\Omega^{\varepsilon}}\left(\nabla\tilde{\Phi}_{1}^{\varepsilon}-\left|Y_{l}\right|^{-1}\mathbb{D}\nabla\tilde{\Phi}_{0}\right)\cdot\nabla\varphi_{2}dx\le C\varepsilon.
\]
This step completes the estimates for $\mathcal{J}_{1}$. More precisely, we obtain
\begin{equation}
	\mathcal{J}_{1}\ge C\left\Vert \nabla\left(\tilde{\Phi}_{\varepsilon}-\tilde{\Phi}_{1}^{\varepsilon}\right)\right\Vert _{\left[L^{2}\left(\Omega^{\varepsilon}\right)\right]^{d}}^{2}-C\left(\varepsilon^{2}+\varepsilon\right).\label{eq:gi1-6}
\end{equation}

In the same vein, we can estimate the term $\mathcal{K}_{2}$ with
the aid of the \emph{a priori} regularity $c_{0}^{\pm}\in W^{1,\infty}\left(\Omega^{\varepsilon}\right)\cap H^{2}\left(\Omega^{\varepsilon}\right)$
and $\varphi_{j}\in W^{1+s,2}\left(Y_{l}\right)$ for $s>d/2$ and
$1\le j\le d$. We thus get
\begin{equation}
	\mathcal{K}_{2}\ge C\left\Vert \nabla\left(c_{\varepsilon}^{\pm}-c_{1}^{\pm,\varepsilon}\right)\right\Vert _{\left[L^{2}\left(\Omega^{\varepsilon}\right)\right]^{d}}^{2}-C\left(\varepsilon^{2}+\varepsilon\right).\label{eq:ka2-6}
\end{equation}

We now turn our attention to the estimates for $\mathcal{J}_{2}$
and $\mathcal{J}_{3}$. Noticing $\bar{\sigma}:=\int_{\Gamma}\sigma dS_{y}$
which implies that
\[
\left|Y_{l}\right|^{-1}\int_{Y_{l}}\bar{\sigma}dy=\int_{\Gamma}\sigma dS_{y},
\]
we then apply \cite[Lemma 5.2]{Tycho} to gain
\[
\left|\mathcal{J}_{2}\right|\le C\varepsilon\left\Vert \varphi_{2}\right\Vert _{H^{1}\left(\Omega^{\varepsilon}\right)}.
\]

Note that due to the choice of $\varphi_{2}$ in (\ref{eq:choice1}),
we have
\begin{align}
	\left\Vert \varphi_{2}\right\Vert _{H^{1}\left(\Omega^{\varepsilon}\right)} & \le\left\Vert \tilde{\Phi}_{\varepsilon}-\tilde{\Phi}_{0}\right\Vert _{L^{2}\left(\Omega^{\varepsilon}\right)}+\left\Vert \nabla\left(\tilde{\Phi}_{\varepsilon}-\tilde{\Phi}_{1}^{\varepsilon}\right)\right\Vert _{\left[L^{2}\left(\Omega^{\varepsilon}\right)\right]^{d}}
	\nonumber \\ &
	+\left\Vert \nabla\left(\tilde{\Phi}_{1}^{\varepsilon}-\tilde{\Phi}_{0}\right)\right\Vert _{\left[L^{2}\left(\Omega^{\varepsilon}\right)\right]^{d}}+\varepsilon\left\Vert m^{\varepsilon}\bar{\varphi}\cdot\nabla_{x}\tilde{\Phi}_{0}\right\Vert _{H^{1}\left(\Omega^{\varepsilon}\right)}\nonumber \\
	& \le\left\Vert \tilde{\Phi}_{\varepsilon}-\tilde{\Phi}_{0}\right\Vert _{L^{2}\left(\Omega^{\varepsilon}\right)}+\left\Vert \nabla\left(\tilde{\Phi}_{\varepsilon}-\tilde{\Phi}_{1}^{\varepsilon}\right)\right\Vert _{\left[L^{2}\left(\Omega^{\varepsilon}\right)\right]^{d}}+C\left(1+\varepsilon+\varepsilon^{\frac{1}{2}}\right),\label{eq:3.10-6-1}
\end{align}
where we use the inequalities (\ref{eq:cutin}) with the regularity
assumptions on $\bar{\varphi}$ and $\tilde{\Phi}_{0}$, and the following
bound:
\[
\left\Vert \nabla\left(\tilde{\Phi}_{1}^{\varepsilon}-\tilde{\Phi}_{0}\right)\right\Vert _{\left[L^{2}\left(\Omega^{\varepsilon}\right)\right]^{d}}\le\left\Vert \nabla_{y}\bar{\varphi}\right\Vert _{C\left(\bar{Y}_{l}\right)}\left\Vert \tilde{\Phi}_{0}\right\Vert _{W^{1,\infty}\left(\Omega^{\varepsilon}\right)}+\varepsilon\left\Vert \bar{\varphi}\right\Vert _{C\left(\bar{Y}_{l}\right)}\left\Vert \tilde{\Phi}_{0}\right\Vert _{H^{2}\left(\Omega^{\varepsilon}\right)}.
\]

Therefore, we can write that
\begin{equation}
	\left|\mathcal{J}_{2}\right|\le C\varepsilon\left(\left\Vert \tilde{\Phi}_{\varepsilon}-\tilde{\Phi}_{0}\right\Vert _{L^{2}\left(\Omega^{\varepsilon}\right)}+\left\Vert \nabla\left(\tilde{\Phi}_{\varepsilon}-\tilde{\Phi}_{1}^{\varepsilon}\right)\right\Vert _{\left[L^{2}\left(\Omega^{\varepsilon}\right)\right]^{d}}+1\right).\label{eq:gi2-6}
\end{equation}

The estimate for $\mathcal{J}_{3}$ can be derived by the H\"older
inequality, which reads
\[
\left|\mathcal{J}_{3}\right|\le C\left(\left\Vert c_{\varepsilon}^{+}-c_{0}^{+}\right\Vert _{L^{2}\left(\Omega^{\varepsilon}\right)}+\left\Vert c_{\varepsilon}^{-}-c_{0}^{-}\right\Vert _{L^{2}\left(\Omega^{\varepsilon}\right)}\right)\left\Vert \varphi_{2}\right\Vert _{L^{2}\left(\Omega^{\varepsilon}\right)},
\]
and then leads to
\begin{equation}
	\left|\mathcal{J}_{3}\right|\le C\left(\left\Vert c_{\varepsilon}^{+}-c_{0}^{+}\right\Vert _{L^{2}\left(\Omega^{\varepsilon}\right)}+\left\Vert c_{\varepsilon}^{-}-c_{0}^{-}\right\Vert _{L^{2}\left(\Omega^{\varepsilon}\right)}\right)\left(\left\Vert \tilde{\Phi}_{\varepsilon}-\tilde{\Phi}_{0}\right\Vert _{L^{2}\left(\Omega^{\varepsilon}\right)}+1\right).\label{eq:gi3-6}
\end{equation}

Let us now consider the term $\mathcal{K}_{1}$ and $\mathcal{K}_{4}$.
Note that $\mathcal{K}_{1}$ can be rewritten as
\begin{align}\int_{\Omega^{\varepsilon}}\partial_{t}\left(c_{\varepsilon}^{\pm}-c_{0}^{\pm}\right)\left[c_{\varepsilon}^{\pm}\right. & \left.-\left(c_{0}^{\pm,\varepsilon}\left(t,x\right)+\varepsilon m^{\varepsilon}\bar{\varphi}^{\varepsilon}\cdot\nabla_{x}c_{0}^{\pm}\right)\right]dx \nonumber \\
 & = \frac{1}{2}\frac{d}{dt}\left\Vert c_{\varepsilon}^{\pm}-c_{0}\right\Vert _{L^{2}\left(\Omega^{\varepsilon}\right)}^{2}-\varepsilon\int_{\Omega^{\varepsilon}}\partial_{t}\left(c_{\varepsilon}^{\pm}-c_{0}^{\pm}\right)m^{\varepsilon}\bar{\varphi}\cdot\nabla_{x}c_{0}^{\pm}dx,
\label{eq:3.10-6}
\end{align}
while from the structure of the reaction in $\left(\mbox{A}_{3}\right)$,
we have the similar result for $\mathcal{K}_{4}$ (to $\mathcal{J}_{3}$),
i.e.
\begin{equation}
	\left|\mathcal{K}_{4}\right|\le C\left(\left\Vert c_{\varepsilon}^{+}-c_{0}^{+}\right\Vert _{L^{2}\left(\Omega^{\varepsilon}\right)}+\left\Vert c_{\varepsilon}^{-}-c_{0}^{-}\right\Vert _{L^{2}\left(\Omega^{\varepsilon}\right)}\right)\left(\left\Vert c_{\varepsilon}^{\pm}-c_{0}^{\pm}\right\Vert _{L^{2}\left(\Omega^{\varepsilon}\right)}+1\right).\label{eq:ka4-6}
\end{equation}

The estimate for $\mathcal{K}_{3}$ relies on the following decomposition:
\begin{align*}
	\left|Y_{l}\right|^{-1}c_{0}^{\pm}\left(\bar{v}_{0}\mp\mathbb{D}\nabla\tilde{\Phi}_{0}\right)-c_{\varepsilon}^{\pm}\left(v_{\varepsilon}\mp\nabla\tilde{\Phi}_{\varepsilon}\right) & =\left(c_{0}^{\pm}-c_{\varepsilon}^{\pm}\right)\left(\left|Y_{l}\right|^{-1}\bar{v}_{0}\mp\left|Y_{l}\right|^{-1}\mathbb{D}\nabla\tilde{\Phi}_{0}\right)\\
	& +c_{\varepsilon}^{\pm}\left(\left|Y_{l}\right|^{-1}\bar{v}_{0}-v_{\varepsilon}\right)\mp c_{\varepsilon}^{\pm}\left(\left|Y_{l}\right|^{-1}\mathbb{D}\nabla\tilde{\Phi}_{0}-\nabla\tilde{\Phi}_{\varepsilon}\right).
\end{align*}

Clearly, if $\bar{v}_{0}\in L^{\infty}\left(\Omega^{\varepsilon}\right)$
and since $\tilde{\Phi}_{0}\in W^{1,\infty}\left(\Omega^{\varepsilon}\right)\cap H^{2}\left(\Omega^{\varepsilon}\right)$,
we can estimate, by H\"older's inequality, that
\begin{equation}
	\int_{\Omega^{\varepsilon}}\left(c_{0}^{\pm}-c_{\varepsilon}^{\pm}\right)\left(\left|Y_{l}\right|^{-1}\bar{v}_{0}\mp\left|Y_{l}\right|^{-1}\mathbb{D}\nabla\tilde{\Phi}_{0}\right)\cdot\nabla\varphi_{3}dx\le C\left\Vert c_{\varepsilon}^{\pm}-c_{0}^{\pm}\right\Vert _{L^{2}\left(\Omega^{\varepsilon}\right)}\left\Vert \nabla\varphi_{3}\right\Vert _{\left[L^{2}\left(\Omega^{\varepsilon}\right)\right]^{d}}.\label{eq:3.12-6}
\end{equation}

By using the same arguments in estimating the norm $\left\Vert \varphi_{2}\right\Vert _{H^{1}\left(\Omega^{\varepsilon}\right)}$
in (\ref{eq:3.10-6-1}), we get from (\ref{eq:3.12-6})
that
\begin{align}
\int_{\Omega^{\varepsilon}}\left(c_{0}^{\pm}-c_{\varepsilon}^{\pm}\right)\left(\left|Y_{l}\right|^{-1}\bar{v}_{0}\right. & \left.\mp\left|Y_{l}\right|^{-1}\mathbb{D}\nabla\tilde{\Phi}_{0}\right)\cdot\nabla\varphi_{3}dx \nonumber\\
& \le C\left\Vert c_{\varepsilon}^{\pm}-c_{0}^{\pm}\right\Vert _{L^{2}\left(\Omega^{\varepsilon}\right)}\left(\left\Vert \nabla\left(c_{\varepsilon}^{\pm}-c_{1}^{\pm,\varepsilon}\right)\right\Vert _{\left[L^{2}\left(\Omega^{\varepsilon}\right)\right]^{d}}+1\right).\label{eq:3.18-6}
\end{align}

Next, we observe that
\begin{align}
	\int_{\Omega^{\varepsilon}}c_{\varepsilon}^{\pm}\left(\left|Y_{l}\right|^{-1}\bar{v}_{0}-v_{\varepsilon}\right)\cdot\nabla\varphi_{3}dx & \le C\left\Vert v_{\varepsilon}-\bar{v}_{0}^{\varepsilon}\right\Vert _{L^{2}\left(\Omega^{\varepsilon}\right)}\left(\left\Vert \nabla\left(c_{\varepsilon}^{\pm}-c_{1}^{\pm,\varepsilon}\right)\right\Vert _{\left[L^{2}\left(\Omega^{\varepsilon}\right)\right]^{d}}+1\right)\nonumber \\
	& \le C\varepsilon^{\frac{1}{2}}\left(\left\Vert \nabla\left(c_{\varepsilon}^{\pm}-c_{1}^{\pm,\varepsilon}\right)\right\Vert _{\left[L^{2}\left(\Omega^{\varepsilon}\right)\right]^{d}}+1\right),\label{eq:3.19-6}
\end{align}
which is a direct result of (\ref{eq:estimate1-6}) and of the fact
that all the microscopic solutions are bounded from above uniformly
in the choice of $\varepsilon$ (see Theorem \ref{thm:positive+bound}).

Using again Theorem \ref{thm:positive+bound},
we estimate that
\begin{align}\int_{\Omega^{\varepsilon}}c_{\varepsilon}^{\pm} & \left(\left|Y_{l}\right|^{-1}\mathbb{D}\nabla\tilde{\Phi}_{0}-\nabla\tilde{\Phi}_{\varepsilon}\right)\cdot\nabla\varphi_{3}dx \nonumber\\
& \le C\left(\left\Vert \nabla\left(\tilde{\Phi}_{\varepsilon}-\tilde{\Phi}_{1}^{\varepsilon}\right)\right\Vert _{\left[L^{2}\left(\Omega^{\varepsilon}\right)\right]^{d}}+\left\Vert \nabla\left(\tilde{\Phi}_{1}^{\varepsilon}-\left|Y_{l}\right|^{-1}\mathbb{D}\tilde{\Phi}_{0}\right)\right\Vert _{\left[L^{2}\left(\Omega^{\varepsilon}\right)\right]^{d}}\right)\nonumber\\
& \times\left(\left\Vert \nabla\left(c_{\varepsilon}^{\pm}-c_{1}^{\pm,\varepsilon}\right)\right\Vert _{\left[L^{2}\left(\Omega^{\varepsilon}\right)\right]^{d}}+1\right)\nonumber\\
& \le C\left(\left\Vert \nabla\left(\tilde{\Phi}_{\varepsilon}-\tilde{\Phi}_{1}^{\varepsilon}\right)\right\Vert _{\left[L^{2}\left(\Omega^{\varepsilon}\right)\right]^{d}}+\varepsilon\right)\left(\left\Vert \nabla\left(c_{\varepsilon}^{\pm}-c_{1}^{\pm,\varepsilon}\right)\right\Vert _{\left[L^{2}\left(\Omega^{\varepsilon}\right)\right]^{d}}+1\right),\label{eq:3.20-6}
\end{align}
which also completes the estimates for $\mathcal{K}_{3}$.

Combining (\ref{eq:gi1-6}), (\ref{eq:ka2-6}), (\ref{eq:gi2-6}),
(\ref{eq:gi3-6}), (\ref{eq:ka4-6}), (\ref{eq:3.18-6}), (\ref{eq:3.19-6})
and (\ref{eq:3.20-6}), we obtain, after some rearrangements, that
\begin{align}
	\left\Vert \nabla\left(\tilde{\Phi}_{\varepsilon}-\tilde{\Phi}_{1}^{\varepsilon}\right)\right\Vert _{\left[L^{2}\left(\Omega^{\varepsilon}\right)\right]^{d}}^{2}&
	+\varepsilon\left\Vert \nabla\left(c_{\varepsilon}^{\pm}-c_{1}^{\pm,\varepsilon}\right)\right\Vert _{\left[L^{2}\left(\Omega^{\varepsilon}\right)\right]^{d}}^{2} \nonumber\\ & \le C\left(\varepsilon^{2}+\varepsilon\right)+C\varepsilon^{\frac{3}{2}}\left(\left\Vert \nabla\left(c_{\varepsilon}^{\pm}-c_{1}^{\pm,\varepsilon}\right)\right\Vert _{\left[L^{2}\left(\Omega^{\varepsilon}\right)\right]^{d}}+1\right)\nonumber \\
	& +C\varepsilon\left(\left\Vert \tilde{\Phi}_{\varepsilon}-\tilde{\Phi}_{0}\right\Vert _{L^{2}\left(\Omega^{\varepsilon}\right)}+\left\Vert \nabla\left(\tilde{\Phi}_{\varepsilon}-\tilde{\Phi}_{1}^{\varepsilon}\right)\right\Vert _{\left[L^{2}\left(\Omega^{\varepsilon}\right)\right]^{d}}\right)\nonumber \\
	& +C\left\Vert c_{\varepsilon}^{\pm}-c_{0}^{\pm}\right\Vert _{L^{2}\left(\Omega^{\varepsilon}\right)}\left(\left\Vert \tilde{\Phi}_{\varepsilon}-\tilde{\Phi}_{0}\right\Vert _{L^{2}\left(\Omega^{\varepsilon}\right)}+1\right)\nonumber \\
	& +C\varepsilon\left(\left\Vert \nabla\left(\tilde{\Phi}_{\varepsilon}-\tilde{\Phi}_{1}^{\varepsilon}\right)\right\Vert _{\left[L^{2}\left(\Omega^{\varepsilon}\right)\right]^{d}}+\varepsilon\right)\left(\left\Vert \nabla\left(c_{\varepsilon}^{\pm}-c_{1}^{\pm,\varepsilon}\right)\right\Vert _{\left[L^{2}\left(\Omega^{\varepsilon}\right)\right]^{d}}+1\right)\nonumber \\
	& +C\varepsilon\left\Vert c_{\varepsilon}^{\pm}-c_{0}^{\pm}\right\Vert _{L^{2}\left(\Omega^{\varepsilon}\right)}\left(\left\Vert \nabla\left(c_{\varepsilon}^{\pm}-c_{1}^{\pm,\varepsilon}\right)\right\Vert _{\left[L^{2}\left(\Omega^{\varepsilon}\right)\right]^{d}}+1\right).\label{eq:3.21-6}
\end{align}

It now remains to estimate the second term on the right-hand side
of (\ref{eq:3.10-6}). In fact, integrating the right-hand side of (\ref{eq:3.10-6}) by parts gives
\begin{align*}
	\int_{0}^{t}\int_{\Omega^{\varepsilon}}m^{\varepsilon}\partial_{t}\left(c_{\varepsilon}^{\pm}-c_{0}^{\pm}\right)\bar{\varphi}\cdot\nabla_{x}c_{0}^{\pm}dxds & =\int_{\Omega^{\varepsilon}}m^{\varepsilon}\left.\left(c_{\varepsilon}^{\pm}-c_{0}^{\pm}\right)\bar{\varphi}\cdot\nabla_{x}c_{0}^{\pm}dx\right|_{s=0}^{s=t}\\
	& -\int_{0}^{t}\int_{\Omega^{\varepsilon}}m^{\varepsilon}\left(c_{\varepsilon}^{\pm}-c_{0}^{\pm}\right)\bar{\varphi}\cdot\nabla_{x}\partial_{t}c_{0}^{\pm}dxds,
\end{align*}
and we also have
\begin{align}
\varepsilon\left|\int_{\Omega^{\varepsilon}}m^{\varepsilon}\left[\left(c_{\varepsilon}^{\pm}-c_{0}^{\pm}\right)\right.\right. & -\left.\left.\left(c_{\varepsilon}^{\pm}\left(0\right)-c_{0}^{\pm}\left(0\right)\right)\right]\bar{\varphi}\cdot\nabla_{x}c_{0}^{\pm}dx\right|\nonumber \\
& \le C\varepsilon\left(\left\Vert c_{\varepsilon}^{\pm}-c_{0}^{\pm}\right\Vert _{L^{2}\left(\Omega^{\varepsilon}\right)}+\left\Vert c_{\varepsilon}^{\pm,0}-c_{0}^{\pm,0}\right\Vert _{L^{2}\left(\Omega^{\varepsilon}\right)}\right).
\label{eq:3.22-6}
\end{align}
At this moment, if we set
\begin{align*}
	w_{1}\left(t\right) & =\left\Vert \tilde{\Phi}_{\varepsilon}\left(t\right)-\tilde{\Phi}_{0}\left(t\right)\right\Vert _{L^{2}\left(\Omega^{\varepsilon}\right)}^{2}+\left\Vert c_{\varepsilon}^{\pm}\left(t\right)-c_{0}^{\pm}\left(t\right)\right\Vert _{L^{2}\left(\Omega^{\varepsilon}\right)}^{2},\\
	w_{2}\left(t\right) & =\left\Vert \nabla\left(\tilde{\Phi}_{\varepsilon}-\tilde{\Phi}_{1}^{\varepsilon}\right)\left(t\right)\right\Vert _{\left[L^{2}\left(\Omega^{\varepsilon}\right)\right]^{d}}^{2}+\varepsilon\left\Vert \nabla\left(c_{\varepsilon}^{\pm}-c_{1}^{\pm,\varepsilon}\right)\left(t\right)\right\Vert _{\left[L^{2}\left(\Omega^{\varepsilon}\right)\right]^{d}}^{2},\\
	w_{0} & =\left\Vert c_{\varepsilon}^{\pm,0}-c_{0}^{\pm,0}\right\Vert _{L^{2}\left(\Omega^{\varepsilon}\right)}^{2},
\end{align*}
then, after integrating (\ref{eq:3.21-6}) and (\ref{eq:3.10-6}) from
0 to $t$, we are led to the following Gronwall-like estimate:
\[
w_{1}\left(t\right)+\int_{0}^{t}w_{2}\left(s\right)ds\le C\left(\varepsilon+\left(1+\varepsilon\right)w_{0}+\int_{0}^{t}w_{1}\left(s\right)ds\right),
\]
which provides that
\[
w_{1}\left(t\right)+\int_{0}^{t}w_{2}\left(s\right)ds\le C\left(\varepsilon+\left(1+\varepsilon\right)w_{0}\right)\quad\text{for }t\in\left[0,T\right].
\]

Assuming
\begin{align}\label{initialguess-6}
\left\Vert c_{\varepsilon}^{\pm,0}-c_{0}^{\pm,0}\right\Vert _{L^{2}\left(\Omega^{\varepsilon}\right)}^{2}\le C\varepsilon^{\mu}\quad\text{for some }\mu\in\mathbb{R}_{+},
\end{align}
we thus obtain
\begin{align}
	& \left\Vert \tilde{\Phi}_{\varepsilon}-\tilde{\Phi}_{0}\right\Vert _{L^{2}\left(\left(0,T\right)\times\Omega^{\varepsilon}\right)}^{2}+\left\Vert c_{\varepsilon}^{\pm}-c_{0}^{\pm}\right\Vert _{L^{2}\left(\left(0,T\right)\times\Omega^{\varepsilon}\right)}^{2}+\left\Vert \nabla\left(\tilde{\Phi}_{\varepsilon}-\tilde{\Phi}_{1}^{\varepsilon}\right)\right\Vert _{\left[L^{2}\left(\left(0,T\right)\times\Omega^{\varepsilon}\right)\right]^{d}}^{2}\nonumber \\
	& +\varepsilon\left\Vert \nabla\left(c_{\varepsilon}^{\pm}-c_{1}^{\pm,\varepsilon}\right)\right\Vert _{\left[L^{2}\left(\left(0,T\right)\times\Omega^{\varepsilon}\right)\right]^{d}}^{2}\le C\max\left\{ \varepsilon,\varepsilon^{\mu}\right\} .\label{eq:estimate2-6}
\end{align}
Since the obtained estimate for $\left\Vert \nabla\left(\tilde{\Phi}_{\varepsilon}-\tilde{\Phi}_{1}^{\varepsilon}\right)\right\Vert _{\left[L^{2}\left(\left(0,T\right)\times\Omega^{\varepsilon}\right)\right]^{d}}$
is of the order of $\mathcal{O}\left(\max\left\{ \varepsilon,\varepsilon^{\mu}\right\} \right)$,
we can also increase the rate of $\left\Vert \nabla\left(c_{\varepsilon}^{\pm}-c_{1}^{\pm,\varepsilon}\right)\right\Vert _{\left[L^{2}\left(\left(0,T\right)\times\Omega^{\varepsilon}\right)\right]^{d}}$.
Indeed, let us consider the estimate (\ref{eq:3.18-6}) and (\ref{eq:3.20-6})
for $\left\Vert c_{\varepsilon}^{\pm}-c_{0}^{\pm}\right\Vert _{L^{2}\left(\left(0,T\right)\times\Omega^{\varepsilon}\right)}$
and $\left\Vert \nabla\left(\tilde{\Phi}_{\varepsilon}-\tilde{\Phi}_{1}^{\varepsilon}\right)\right\Vert _{\left[L^{2}\left(\left(0,T\right)\times\Omega^{\varepsilon}\right)\right]^{d}}$,
respectively. Then, we combine again (\ref{eq:ka2-6}), (\ref{eq:ka4-6}),
(\ref{eq:3.18-6}), (\ref{eq:3.19-6}), (\ref{eq:3.20-6}) and (\ref{eq:3.22-6})
to get another Gronwall-like estimate:
\[
\left\Vert \nabla\left(c_{\varepsilon}^{\pm}-c_{1}^{\pm,\varepsilon}\right)\left(t\right)\right\Vert _{\left[L^{2}\left(\Omega^{\varepsilon}\right)\right]^{d}}^{2}\le C\left(\varepsilon^{\frac{1}{2}}+\max\left\{ \varepsilon,\varepsilon^{\mu}\right\} +\varepsilon\int_{0}^{t}\left\Vert \nabla\left(c_{\varepsilon}^{\pm}-c_{1}^{\pm,\varepsilon}\right)\left(s\right)\right\Vert _{\left[L^{2}\left(\Omega^{\varepsilon}\right)\right]^{d}}^{2}ds\right).
\]

As a result, we have
\begin{equation}
	\left\Vert \nabla\left(c_{\varepsilon}^{\pm}-c_{1}^{\pm,\varepsilon}\right)\right\Vert _{\left[L^{2}\left(\left(0,T\right)\times\Omega^{\varepsilon}\right)\right]^{d}}^{2}\le C\max\left\{ \varepsilon^{\frac{1}{2}},\varepsilon^{\mu}\right\} .\label{eq:estimate3-6}
\end{equation}

Note that for $\gamma>\alpha$, the drift term in the macroscopic
Nernst-Planck system is not present. Thus, this term does not appear in (\ref{eq:3.18-6})
and (\ref{eq:3.20-6}). Due to the \emph{a priori} estimate that $\left\Vert \tilde{\Phi}_{\varepsilon}\right\Vert _{L^{2}\left(0,T;H^{1}\left(\Omega^{\varepsilon}\right)\right)}\le C$
(cf. Theorem \ref{thm:aprioriestimate-N}) in
combination with the boundedness of $c_{\varepsilon}^{\pm}$ (cf. Theorem
\ref{thm:positive+bound}), it is straightforward to get the same corrector
estimate as (\ref{eq:estimate2-6}). Moreover, if $\alpha<0$, the
corrector becomes of the order $\mathcal{O}\left(\max\left\{ \varepsilon,\varepsilon^{-\alpha},\varepsilon^{\mu}\right\} \right)$. This explicitly illustrates the effect of the scaling parameter $\alpha$ on the
rate of the convergence.

For the time being, it only remains to come up with the corrector estimates for the Stokes equation. At this point, we must pay a regularity price\footnote{Compare to the two-scale convergence
	method when deriving the structure of the macroscopic system in \cite{RMK12}.} 
concerning the smoothness of the boundaries to make use of Lemma \ref{lem:ines2}.
With $\partial\Omega\in C^{4}$, we adapt the ideas of \cite{PM96} to  
define the following velocity corrector:
\begin{align}
	\mathcal{V}^{\varepsilon,\delta}\left(t,x\right) & :=-\sum_{j=1}^{d}w_{j}\left(\frac{x}{\varepsilon}\right)\left[\left(c_{0}^{+}-c_{0}^{-}\right)\partial_{x_{j}}\tilde{\Phi}_{0}\left(t,x\right)+\partial_{x_{j}}p_{0}\left(t,x\right)+\left(\mathbb{K}^{-1}\eta^{\delta}\right)_{j}\right]\nonumber \\
	& -\varepsilon\sum_{i,j=1}^{d}r_{ij}\left(\frac{x}{\varepsilon}\right)\left(1-m^{\varepsilon}\right)\partial_{x_{i}}\left(\left(c_{0}^{+}-c_{0}^{-}\right)\partial_{x_{j}}\tilde{\Phi}_{0}\left(t,x\right)+\partial_{x_{j}}p_{0}\left(t,x\right)+\left(\mathbb{K}^{-1}\eta^{\delta}\right)_{j}\right),\label{eq:Unghieng}
\end{align}
and the pressure corrector:
\begin{align}
	\mathcal{P}^{\varepsilon,\delta}\left(t,x\right)&:=p_{0}\left(t,x\right) \nonumber \\&-\varepsilon\sum_{j=1}^{d}\pi_{j}\left(\frac{x}{\varepsilon}\right)\left[\left(c_{0}^{+}-c_{0}^{-}\right)\partial_{x_{j}}\tilde{\Phi}_{0}\left(t,x\right)+\partial_{x_{j}}p_{0}\left(t,x\right)+\left(\mathbb{K}^{-1}\eta^{\delta}\right)_{j}\right],\label{eq:Pnghieng}
\end{align}
where $w_{j}$, $\pi_{j}$ and $r_{ij}$ are solutions of the problems
(\ref{eq:cellproblem1}) and (\ref{eq:lastcellproblem}), respectively,
for $1\le i,j\le d$; and $\eta^{\delta}$ is a function defined in
Lemma \ref{lem:ines2}.

From (\ref{eq:Unghieng}), one can structure the divergence of the
corrector $\mathcal{V}^{\varepsilon,\delta}$. In fact, by definition
of the function $\eta^{\delta}$ and the structure of the macroscopic
system for the velocity in Theorem \ref{thm:macroNeumann}, the divergence
of the first term of vanishes (\ref{eq:Unghieng}) itself. Therefore,
one computes that
\begin{align*}
	\nabla\cdot\mathcal{V}^{\varepsilon,\delta} & =-\sum_{i,j=1}^{d}\left(w_{j}^{i}\left(\frac{x}{\varepsilon}\right)-\left|Y_{l}\right|^{-1}K_{ij}\right)\left(1-m^{\varepsilon}\right)\partial_{x_{i}}\left[\left(c_{0}^{+}-c_{0}^{-}\right)\partial_{x_{j}}\tilde{\Phi}_{0}+\partial_{x_{j}}p_{0}+\left(\mathbb{K}^{-1}\eta^{\delta}\right)_{j}\right]\\
	& -\varepsilon\sum_{i,j=1}^{d}r_{ij}\left(\frac{x}{\varepsilon}\right)\left(1-m^{\varepsilon}\right)\nabla\cdot\left[\partial_{x_{i}}\left(\left(c_{0}^{+}-c_{0}^{-}\right)\partial_{x_{j}}\tilde{\Phi}_{0}+\partial_{x_{j}}p_{0}+\left(\mathbb{K}^{-1}\eta^{\delta}\right)_{j}\right)\right]\\
	& +\varepsilon\sum_{i,j=1}^{d}r_{ij}\left(\frac{x}{\varepsilon}\right)\nabla m^{\varepsilon}\partial_{x_{i}}\left[\left(c_{0}^{+}-c_{0}^{-}\right)\partial_{x_{j}}\tilde{\Phi}_{0}+\partial_{x_{j}}p_{0}+\left(\mathbb{K}^{-1}\eta^{\delta}\right)_{j}\right],
\end{align*}
where we also use the structure of the cell problem (\ref{eq:lastcellproblem}).

Taking into account that
\begin{align*}
	& -\sum_{i,j=1}^{d}K_{ij}\partial_{x_{i}}\left(\left(c_{0}^{\pm}-c_{0}^{-}\right)\partial_{x_{j}}\tilde{\Phi}_{0}+\partial_{x_{j}}p_{0}\right)=0,\\
	& \sum_{i,j=1}^{d}K_{ij}\partial_{x_{i}}\left(\mathbb{K}^{-1}\eta^{\delta}\right)_{j}=0,
\end{align*}
hold (see again the macroscopic system for the velocity in Theorem
\ref{thm:macroNeumann} as well as the properties of $\eta^{\delta}$
in Lemma \ref{lem:ines2}), the estimate for the divergence of $\mathcal{V}^{\varepsilon,\delta}$
in $L^{2}$-norm
\[
\left\Vert \nabla\cdot\mathcal{V}^{\varepsilon,\delta}\right\Vert _{L^{2}\left(\Omega^{\varepsilon}\right)}\le C\left(\varepsilon^{\frac{1}{2}}\delta^{-1}+\varepsilon\delta^{-\frac{3}{2}}+\varepsilon^{\frac{1}{q}}\delta^{-\frac{1}{2}-\frac{1}{q}}\right)\quad\text{for }q\in\left[2,\infty\right],
\]
is directly obtained from Lemma \ref{lem:ines2} and the inequalities
in (\ref{eq:cutin}).

At this stage, if we choose $q=2$ and $\delta\gg\varepsilon$, we
get
\begin{equation}
	\left\Vert \nabla\cdot\mathcal{V}^{\varepsilon,\delta}\right\Vert _{L^{2}\left(\Omega^{\varepsilon}\right)}\le C\left(\varepsilon\delta^{-\frac{3}{2}}+\varepsilon^{\frac{1}{2}}\delta^{-1}\right),\label{eq:3.28-6}
\end{equation}
and hence,
\[
\left\Vert \nabla\cdot\mathcal{V}^{\varepsilon,\delta}\right\Vert _{L^{2}\left(\left(0,T\right)\times\Omega^{\varepsilon}\right)}\le C\left(\varepsilon\delta^{-\frac{3}{2}}+\varepsilon^{\frac{1}{2}}\delta^{-1}\right).
\]

Next, we introduce the following function:
\[
\Psi^{\varepsilon}\left(t,x\right):=\Delta\mathcal{V}^{\varepsilon,\delta}\left(t,x\right)-\varepsilon^{-2}\nabla\mathcal{P}^{\varepsilon,\delta}-\left(c_{0}^{+}\left(t,x\right)-c_{0}^{-}\left(t,x\right)\right)\nabla\tilde{\Phi}_{0}\left(t,x\right).
\]

Thus, for any $\varphi_{1}\in\left[H_{0}^{1}\left(\Omega^{\varepsilon}\right)\right]^{d}$
we have, after direct computations, that
\begin{align} & \left\langle \Psi^{\varepsilon},\varphi_{1}\right\rangle _{\left(\left[H^{1}\right]^{d}\right)',\left[H^{1}\right]^{d}}\\
= & -\sum_{j=1}^{d}\int_{\Omega^{\varepsilon}}\left(\Delta w_{j}\left(\frac{x}{\varepsilon}\right)-\varepsilon^{-1}\nabla\pi_{j}\left(\frac{x}{\varepsilon}\right)\right)\left[\left(c_{0}^{+}-c_{0}^{-}\right)\partial_{x_{j}}\tilde{\Phi}_{0}+\partial_{x_{j}}p_{0}+\left(\mathbb{K}^{-1}\eta^{\delta}\right)_{j}\right]\varphi_{1}dx\nonumber\\
& -\varepsilon^{-2}\int_{\Omega^{\varepsilon}}\left(\nabla p_{0}+\left(c_{0}^{+}-c_{0}^{-}\right)\nabla\tilde{\Phi}_{0}\right)\varphi_{1}dx\nonumber\\
& -\sum_{j=1}^{d}\int_{\Omega^{\varepsilon}}\left(2\nabla w_{j}\left(\frac{x}{\varepsilon}\right)-\varepsilon^{-1}\pi\left(\frac{x}{\varepsilon}\right)\mathbb{I}\right)\nabla\left[\left(c_{0}^{+}-c_{0}^{-}\right)\partial_{x_{j}}\tilde{\Phi}_{0}+\partial_{x_{j}}p_{0}+\left(\mathbb{K}^{-1}\eta^{\delta}\right)_{j}\right]\varphi_{1}dx\nonumber\\
& -\sum_{j=1}^{d}\int_{\Omega^{\varepsilon}}w_{j}\left(\frac{x}{\varepsilon}\right)\Delta\left[\left(c_{0}^{+}-c_{0}^{-}\right)\partial_{x_{j}}\tilde{\Phi}_{0}+\partial_{x_{j}}p_{0}+\left(\mathbb{K}^{-1}\eta^{\delta}\right)_{j}\right]\varphi_{1}dx\nonumber\\
& -\varepsilon\sum_{j=1}^{d}\int_{\Omega^{\varepsilon}}\nabla\left[r_{ij}\left(\frac{x}{\varepsilon}\right)\left(1-m^{\varepsilon}\right)\partial_{x_{j}}\left(\left(c_{0}^{+}-c_{0}^{-}\right)\partial_{x_{j}}\tilde{\Phi}_{0}+\partial_{x_{j}}p_{0}+\left(\mathbb{K}^{-1}\eta^{\delta}\right)_{j}\right)\right]\cdot\nabla\varphi_{1}dx\nonumber\\
& :=\mathcal{I}_{1}+\mathcal{I}_{2}+\mathcal{I}_{3}+\mathcal{I}_{4}+\mathcal{I}_{5}.\label{eq:3.4-6}
\end{align}

Note that $\mathbb{I}$ here stands for the identity matrix. From
now on, to get the estimate for $\Psi^{\varepsilon}$ in $\left(H^{1}\right)'$-norm,
we need bounds on $\mathcal{I}_{i}$ for $1\le i\le5$.
Indeed, with the help of Lemma \ref{lem:ines3}
applied to the test function $\varphi_{1}$, and the estimates of
the involved functions, one immediately obtains from the H\"older's
inequality that
\begin{equation}
	\left|\mathcal{I}_{3}\right|+\left|\mathcal{I}_{4}\right|\le C\left(\delta^{-\frac{1}{2}}+\varepsilon\delta^{-\frac{3}{2}}\right)\left\Vert \nabla\varphi_{1}\right\Vert _{\left[L^{2}\left(\Omega^{\varepsilon}\right)\right]^{d}},\label{eq:3.5-6}
\end{equation}
where we also apply again the estimate of $\eta^{\delta}$ in Lemma
\ref{lem:ines2}.

To estimate  $\mathcal{I}_{5}$, we notice
by
\begin{equation}
	\left|\mathcal{I}_{5}\right|\le C\left(\delta^{-\frac{1}{2}}+\varepsilon\delta^{-\frac{3}{2}}\right)\left\Vert \nabla\varphi_{1}\right\Vert _{\left[L^{2}\left(\Omega^{\varepsilon}\right)\right]^{d}},\label{eq:3.6-6}
\end{equation}
where we also employ the estimates (\ref{eq:cutin}) on $m^{\varepsilon}$.

In addition, we have
\begin{align}\left|\mathcal{I}_{1}+\mathcal{I}_{2}\right| & \le\left|\int_{\Omega^{\varepsilon}}\varepsilon^{-2}\left[-\sum_{j=1}^{d}\left(\left(c_{0}^{+}-c_{0}^{-}\right)\partial_{x_{j}}\tilde{\Phi}_{0}+\partial_{x_{j}}p_{0}+\left(\mathbb{K}^{-1}\eta^{\delta}\right)_{j}\right)\right.\right.\nonumber\\
& +\left.\left.\nabla p_{0}+\left(c_{0}^{+}-c_{0}^{-}\right)\nabla\tilde{\Phi}_{0}\right]\varphi_{1}dx\right|\nonumber\\
& \le C\varepsilon^{-1}\delta^{\frac{1}{2}}\left\Vert \nabla\varphi_{1}\right\Vert _{\left[L^{2}\left(\Omega^{\varepsilon}\right)\right]^{d}}.\label{eq:3.7-6}
\end{align}

Consequently, collecting (\ref{eq:3.4-6})-(\ref{eq:3.7-6}) and according
to the definition of the $\left(H^{1}\right)'$-norm, we arrive at
\begin{align}
	\left\Vert \Psi^{\varepsilon}\right\Vert _{\left(\left[H^{1}\left(\Omega^{\varepsilon}\right)\right]^{d}\right)'} & =\sup_{\varphi_{1}\in\left[H^{1}\left(\Omega^{\varepsilon}\right)\right]^{d},\left\Vert \varphi_{1}\right\Vert _{\left[H^{1}\left(\Omega^{\varepsilon}\right)\right]^{d}}\le1}\left\langle \Psi^{\varepsilon},\varphi_{1}\right\rangle _{\left(\left[H^{1}\right]^{d}\right)',\left[H^{1}\right]^{d}}\nonumber \\
	& \le C\left(\varepsilon^{-1}\delta^{\frac{1}{2}}+\delta^{-\frac{1}{2}}+\varepsilon\delta^{-\frac{3}{2}}\right)\left\Vert \nabla\varphi_{1}\right\Vert _{\left[L^{2}\left(\Omega^{\varepsilon}\right)\right]^{d}}.\label{eq:3.32-6}
\end{align}

Now, we have available a couple of estimates related to the correctors $\mathcal{V}^{\varepsilon,\delta}$
and $\mathcal{P}^{\varepsilon,\delta}$. To go on, we consider
the differences
\[
\mathcal{D}_{1}^{\varepsilon}:=v_{\varepsilon}-\left|Y_{l}\right|^{-1}\mathbb{D}\mathcal{V}^{\varepsilon,\delta},\quad\mathcal{D}_{2}^{\varepsilon}:=p_{\varepsilon}-\left|Y_{l}\right|^{-1}\mathbb{D}\mathcal{P}^{\varepsilon,\delta},
\]
and observe that the equation
\begin{equation}
	-\varepsilon^{2}\Delta\mathcal{D}_{1}^{\varepsilon}+\nabla\mathcal{D}_{2}^{\varepsilon}=\varepsilon^{2}\left[\left|Y_{l}\right|^{-1}\mathbb{D}\Psi^{\varepsilon}-\varepsilon^{-2}\left(\left(c_{\varepsilon}^{+}-c_{\varepsilon}^{-}\right)\nabla\tilde{\Phi}_{\varepsilon}-\left(c_{0}^{+}-c_{0}^{-}\right)\left|Y_{l}\right|^{-1}\mathbb{D}\nabla\tilde{\Phi}_{0}\right)\right]\label{eq:3.39-6-1}
\end{equation}
holds a.e. in $\Omega^{\varepsilon}$.

It remains to estimate the second term on the right-hand side
of the equation (\ref{eq:3.36-6-1}) in $\left(H^{1}\right)'$-norm.
This estimate fully relies on the corrector estimate for the electrostatic
potentials in (\ref{eq:estimate2-6}), the\emph{ }boundedness of concentration
fields in Theorem \ref{thm:positive+bound} with the assumption that
$c_{0}^{\pm}\in W^{1,\infty}\left(\Omega^{\varepsilon}\right)\cap H^{2}\left(\Omega^{\varepsilon}\right)$.
In fact, the estimate resembles very much the one in (\ref{eq:3.20-6}), viz.
\begin{align}\left\langle \left(c_{\varepsilon}^{+}-c_{\varepsilon}^{-}\right)\nabla\tilde{\Phi}_{\varepsilon}\right. & -\left.\left(c_{0}^{+}-c_{0}^{-}\right)\left|Y_{l}\right|^{-1}\mathbb{D}\nabla\tilde{\Phi}_{0},\varphi_{1}\right\rangle _{\left(\left[H^{1}\right]^{d}\right)',\left[H^{1}\right]^{d}} \nonumber \\
& \le C\left\Vert \nabla\tilde{\Phi}_{\varepsilon}-\left|Y_{l}\right|^{-1}\mathbb{D}\nabla\tilde{\Phi}_{0}\right\Vert _{\left[L^{2}\left(\Omega^{\varepsilon}\right)\right]^{d}}\left\Vert \varphi_{1}\right\Vert _{L^{2}\left(\Omega^{\varepsilon}\right)}\nonumber\\
& \le C\max\left\{ \varepsilon^{\frac{3}{2}},\varepsilon^{\frac{\mu}{2}+1}\right\} \left\Vert \nabla\varphi_{1}\right\Vert _{\left[L^{2}\left(\Omega^{\varepsilon}\right)\right]^{d}},\label{eq:3.33-6}
\end{align}
for all $\varphi_{1}\in\left[H_{0}^{1}\left(\Omega^{\varepsilon}\right)\right]^{d}$
and where we also use Lemma \ref{lem:ines3}.

For ease of presentation, we put
\[
\mathcal{L}^{\varepsilon}:=\varepsilon^{-2}\left(\left(c_{\varepsilon}^{+}-c_{\varepsilon}^{-}\right)\nabla\tilde{\Phi}_{\varepsilon}-\left(c_{0}^{+}-c_{0}^{-}\right)\left|Y_{l}\right|^{-1}\mathbb{D}\nabla\tilde{\Phi}_{0}\right).
\]

The corrector for the pressure can be obtained by the use of the following
results which are deduced from \cite{Tartar80} and \cite{PM96}:
\begin{itemize}
	\item there exists an extension $E\left(\mathcal{D}_{2}^{\varepsilon}\right)\in L^{2}\left(\Omega\right)/\mathbb{R}$
	of $\mathcal{D}_{2}^{\varepsilon}$ such that
	\begin{equation}
		\left\Vert E\left(\mathcal{D}_{2}^{\varepsilon}\right)\right\Vert _{L^{2}\left(\Omega\right)/\mathbb{R}}\le C\varepsilon\left(\left\Vert \Psi^{\varepsilon}-\mathcal{L}^{\varepsilon}\right\Vert _{\left(\left[H^{1}\left(\Omega^{\varepsilon}\right)\right]^{d}\right)'}+\left\Vert \nabla\mathcal{D}_{1}^{\varepsilon}\right\Vert _{\left[L^{2}\left(\Omega^{\varepsilon}\right)\right]^{d^{2}}}\right),\label{eq:3.35-6}
	\end{equation}
	\item the following estimates hold:
	\begin{align}
		& \left\Vert \nabla\mathcal{D}_{1}^{\varepsilon}\right\Vert _{\left[L^{2}\left(\Omega^{\varepsilon}\right)\right]^{d^{2}}}\le C\left(\left\Vert \Psi^{\varepsilon}-\mathcal{L}^{\varepsilon}\right\Vert _{\left(\left[H^{1}\left(\Omega^{\varepsilon}\right)\right]^{d}\right)'}+\varepsilon^{-1}\left\Vert \nabla\cdot\mathcal{V}^{\varepsilon,\delta}\right\Vert _{L^{2}\left(\Omega^{\varepsilon}\right)}\right),\label{eq:3.36-6-1}\\
		& \left\Vert \mathcal{D}_{1}^{\varepsilon}\right\Vert _{\left[L^{2}\left(\Omega^{\varepsilon}\right)\right]^{d}}\le C\left(\varepsilon\left\Vert \Psi^{\varepsilon}-\mathcal{L}^{\varepsilon}\right\Vert _{\left(\left[H^{1}\left(\Omega^{\varepsilon}\right)\right]^{d}\right)'}+\left\Vert \nabla\cdot\mathcal{V}^{\varepsilon,\delta}\right\Vert _{L^{2}\left(\Omega^{\varepsilon}\right)}\right).\label{eq:3.36-6}
	\end{align}
\end{itemize}

Collecting (\ref{eq:3.32-6}) and (\ref{eq:3.33-6}),
we get
\begin{equation}
	\left\Vert \Psi^{\varepsilon}-\mathcal{L}^{\varepsilon}\right\Vert _{\left(\left[H^{1}\left(\Omega^{\varepsilon}\right)\right]^{d}\right)'}\le C\left(\varepsilon^{-1}\delta^{\frac{1}{2}}+\delta^{-\frac{1}{2}}+\varepsilon\delta^{-\frac{3}{2}}+\max\left\{ \varepsilon^{-\frac{1}{2}},\varepsilon^{\frac{\mu}{2}-1}\right\} \right)\left\Vert \nabla\varphi_{1}\right\Vert _{\left[L^{2}\left(\Omega^{\varepsilon}\right)\right]^{d}}.\label{eq:3.38-6}
\end{equation}

We thus observe from (\ref{eq:3.36-6}), (\ref{eq:3.28-6}) and (\ref{eq:3.38-6})
that
\[
\left\Vert \mathcal{D}_{1}^{\varepsilon}\right\Vert _{\left[L^{2}\left(\Omega^{\varepsilon}\right)\right]^{d}}\le C\left(\delta^{\frac{1}{2}}+\varepsilon\delta^{-\frac{1}{2}}+\varepsilon^{2}\delta^{-\frac{3}{2}}+\max\left\{ \varepsilon^{\frac{1}{2}},\varepsilon^{\frac{\mu}{2}}\right\} +\varepsilon\delta^{-\frac{3}{2}}+\varepsilon^{\frac{1}{2}}\delta^{-1}\right).
\]

Since $\delta\gg\varepsilon$, we can take $\delta=\varepsilon^{\lambda}$
for $\lambda\in\left(0,1\right)$ to obtain
\begin{align*}
	\left\Vert \mathcal{D}_{1}^{\varepsilon}\right\Vert _{\left[L^{2}\left(\Omega^{\varepsilon}\right)\right]^{d}} & \le C\left(\varepsilon^{\frac{\lambda}{2}}+\varepsilon^{1-\frac{\lambda}{2}}+\varepsilon^{2-\frac{3\lambda}{2}}+\varepsilon^{1-\frac{3\lambda}{2}}+\varepsilon^{\frac{1}{2}-\lambda}+\max\left\{ \varepsilon^{\frac{1}{2}},\varepsilon^{\frac{\mu}{2}}\right\} \right)\\
	& \le C\left(\max\left\{ \varepsilon^{\frac{1}{2}},\varepsilon^{\frac{\mu}{2}}\right\} +\varepsilon^{\frac{\lambda}{2}}+\varepsilon^{1-\frac{3\lambda}{2}}+\varepsilon^{\frac{1}{2}-\lambda}\right).
\end{align*}

On the other hand, the optimal value for $\lambda$ is $1/3$ which
leads to the following estimate:
\begin{equation}
	\left\Vert \mathcal{D}_{1}^{\varepsilon}\right\Vert _{\left[L^{2}\left(\Omega^{\varepsilon}\right)\right]^{d}}\le C\max\left\{ \varepsilon^{\frac{1}{6}},\varepsilon^{\frac{\mu}{2}}\right\} .\label{eq:estimate4-6}
\end{equation}

Hereafter, it follows from (\ref{eq:estimate4-6}), (\ref{eq:3.35-6}),
(\ref{eq:3.36-6-1}) and (\ref{eq:3.38-6}) that
\begin{align*}
	\left\Vert E\left(\mathcal{D}_{2}^{\varepsilon}\right)\right\Vert _{L^{2}\left(\Omega\right)/\mathbb{R}} & \le C\left(\varepsilon\left\Vert \Psi^{\varepsilon}-\mathcal{L}^{\varepsilon}\right\Vert _{\left(\left[H^{1}\left(\Omega^{\varepsilon}\right)\right]^{d}\right)'}+\left\Vert \nabla\cdot\mathcal{V}^{\varepsilon,\delta}\right\Vert _{\left[L^{2}\left(\Omega^{\varepsilon}\right)\right]^{d^{2}}}\right)\\
	& \le C\left(\max\left\{ \varepsilon^{\frac{1}{2}},\varepsilon^{\frac{\mu}{2}}\right\} +\varepsilon^{\frac{\lambda}{2}}+\varepsilon^{1-\frac{3\lambda}{2}}+\varepsilon^{\frac{1}{2}-\lambda}\right).
\end{align*}

This indicates the following estimate:
\begin{equation}
	\left\Vert p_{\varepsilon}-p_{0}\right\Vert _{L^{2}\left(\Omega\right)/\mathbb{R}}\le C\left(\max\left\{ \varepsilon^{\frac{1}{2}},\varepsilon^{\frac{\mu}{2}}\right\} +\varepsilon^{\frac{\lambda}{2}}+\varepsilon^{1-\frac{3\lambda}{2}}+\varepsilon^{\frac{1}{2}-\lambda}\right).\label{eq:estimate5-6}
\end{equation}

Finally, we gather (\ref{eq:estimate1-6}), (\ref{eq:estimate2-6}),
(\ref{eq:estimate3-6}), (\ref{eq:estimate4-6}) and (\ref{eq:estimate5-6})
to conclude the proof of Theorem \ref{thm:main1-6}.

\subsection{Proof of Theorem \ref{thm:main2-6}}

We turn the attention to the Dirichlet boundary condition for the electrostatic potential on the micro-surface. Based on Theorem \ref{thm:macroDirichlet}, we observe that the structure
of the macroscopic systems for the Stokes and Nernst-Planck equations
are the same as the corresponding systems in the Neumann case
(see Theorem \ref{thm:macroNeumann}). Therefore, the corrector estimates
for these systems remain unchanged in Theorem \ref{thm:main1-6}.
Also, some regularity properties are not needed in this case.
We derive first the corrector estimates
for the velocity and pressure and then the corrector estimates of the concentration
fields. Thereby, the corrector for the electrostatic potential can
also be obtained. Here, the macroscopic reconstructions are defined as follows:
\begin{align}
	& v_{0}^{\varepsilon}\left(t,x\right):=v_{0}\left(t,x,\frac{x}{\varepsilon}\right),\label{eq:3.49-6-1}\\
	& v_{1}^{\varepsilon}\left(t,x\right):=v_{1}\left(t,x,\frac{x}{\varepsilon}\right),\label{eq:3.50-6-1}\\
	& c_{0}^{\pm,\varepsilon}\left(t,x\right):=c_{0}^{\pm}\left(t,x\right),\label{eq:3.51-6-1}\\
	& c_{1}^{\pm,\varepsilon}\left(t,x\right):=c_{0}^{\pm,\varepsilon}\left(t,x\right)+\varepsilon\sum_{j=1}^{d}\varphi_{j}\left(\frac{x}{\varepsilon}\right)\partial_{x_{j}}c_{0}^{\pm,\varepsilon}\left(t,x\right).\label{eq:3.52-6-1}
\end{align}

Recall $\tilde{\Phi}_{\varepsilon}:=\varepsilon^{\alpha-2}\Phi_{\varepsilon}^{\hom}$. By Theorem \ref{thm:homoDirichlet}, $\tilde{\Phi}_{\varepsilon}$ obeys the weak formulation
\[
\int_{\Omega^{\varepsilon}}\varepsilon^{2}\nabla\tilde{\Phi}_{\varepsilon}\cdot\nabla\varphi_{2}dx=\int_{\Omega^{\varepsilon}}\left(c_{\varepsilon}^{+}-c_{\varepsilon}^{-}\right)\varphi_{2}dx\quad\text{for all }\varphi_{2}\in H_{0}^{1}\left(\Omega^{\varepsilon}\right).
\]

Therefore, we define the following macroscopic reconstructions:
\begin{align}
	\tilde{\Phi}_{0}^{\varepsilon}\left(t,x\right) & :=\tilde{\Phi}_{0}\left(t,x,\frac{x}{\varepsilon}\right),\label{eq:3.53-6-1}\\
	\overline{\tilde{\Phi}}_{0}^{\varepsilon}\left(t,x\right) & :=\left|Y_{l}\right|^{-1}\overline{\tilde{\Phi}}_{0}\left(t,x\right),\label{eq:3.54-6-1}
\end{align}
and recall that the strong formulation for $\tilde{\Phi}_{0}$
(see \cite[Theorem 4.12]{RMK12}) is given by
\begin{align*}
	& -\Delta_{y}\tilde{\Phi}_{0}\left(t,x,y\right)=c_{0}^{\pm}\left(t,x\right)-c_{0}^{-}\left(t,x\right)\;\text{in }\left(0,T\right)\times\Omega\times Y_{l},\\
	& \tilde{\Phi}_{0}=0\;\text{ in }\left(0,T\right)\times\Omega\times\Gamma.
\end{align*}

Consequently, the difference equation for the Poisson equation can be written
as
\[
-\varepsilon^{2}\Delta\tilde{\Phi}_{\varepsilon}+\left(\Delta_{y}\tilde{\Phi}_{0}\right)^{\varepsilon}=\left(c_{\varepsilon}^{+}-c_{0}^{+}\right)+\left(c_{0}^{-}-c_{\varepsilon}^{-}\right).
\]

Choosing the test function $\varphi_{2}=\tilde{\Phi}_{\varepsilon}-\tilde{\Phi}_{0}^{\varepsilon}$,
let us now estimate the following integral:
\[
\int_{\Omega^{\varepsilon}}\left(\Delta_{y}\tilde{\Phi}_{0}\right)^{\varepsilon}\varphi_{2}dx.
\]

Using the simple relation $\nabla_{y}=\varepsilon\left(\nabla-\nabla_{x}\right)$
and the decomposition
\[
\left(\Delta_{y}\tilde{\Phi}_{0}\right)^{\varepsilon}=\left(1-m^{\varepsilon}\right)\left(\Delta_{y}\tilde{\Phi}_{0}\right)^{\varepsilon}+\varepsilon m^{\varepsilon}\nabla\cdot\left(\nabla_{y}\left(\tilde{\Phi}_{0}\right)^{\varepsilon}\right)-\varepsilon m^{\varepsilon}\left(\nabla_{x}\cdot\left(\nabla_{y}\tilde{\Phi}_{0}\right)\right)^{\varepsilon},
\]
and we obtain, after integrating by parts the term $\nabla\cdot\left(\nabla_{y}\left(\tilde{\Phi}_{0}\right)^{\varepsilon}\right)$,
that
\begin{align}\int_{\Omega^{\varepsilon}}\left(\Delta_{y}\tilde{\Phi}_{0}\right)^{\varepsilon}\varphi_{2}dx & =\int_{\Omega^{\varepsilon}}\left[\left(1-m^{\varepsilon}\right)\left(\Delta_{y}\tilde{\Phi}_{0}\right)^{\varepsilon}\right.\nonumber\\
& -\left.\varepsilon m^{\varepsilon}\left(\nabla_{x}\cdot\left(\nabla_{y}\tilde{\Phi}_{0}\right)\right)^{\varepsilon}-\varepsilon\nabla m^{\varepsilon}\cdot\nabla_{y}\left(\tilde{\Phi}_{0}\right)^{\varepsilon}\right]\varphi_{2}dx\nonumber\\
& +\varepsilon\int_{\Omega^{\varepsilon}}\left(1-m^{\varepsilon}\right)\nabla_{y}\left(\tilde{\Phi}_{0}\right)^{\varepsilon}\cdot\nabla\varphi_{2}dx-\varepsilon\int_{\Omega^{\varepsilon}}\nabla_{y}\left(\tilde{\Phi}_{0}\right)^{\varepsilon}\cdot\nabla\varphi_{2}dx\nonumber\\
& :=\mathcal{F}_{1}+\mathcal{F}_{2}+\mathcal{F}_{3}.\label{eq:3.49-6}
\end{align}

The first and second integrals on the right-hand side of (\ref{eq:3.49-6})
can be estimated by
\begin{align*}\left|\mathcal{F}_{1}\right|+\left|\mathcal{F}_{2}\right| & \le C\left(\left\Vert 1-m^{\varepsilon}\right\Vert _{L^{2}\left(\Omega^{\varepsilon}\right)}\left\Vert \Delta_{y}\tilde{\Phi}_{0}\right\Vert _{L^{\infty}\left(\Omega^{\varepsilon};C\left(Y_{l}\right)\right)}\right.\\
& +\left.\varepsilon\left\Vert \nabla_{x}\cdot\left(\nabla_{y}\tilde{\Phi}_{0}\right)\right\Vert _{L^{2}\left(\Omega^{\varepsilon};C\left(Y_{l}\right)\right)}\right)\left\Vert \varphi_{2}\right\Vert _{L^{2}\left(\Omega^{\varepsilon}\right)}\\
& +C\varepsilon\left\Vert \nabla m^{\varepsilon}\right\Vert _{L^{2}\left(\Omega^{\varepsilon}\right)}\left\Vert \nabla_{y}\tilde{\Phi}_{0}\right\Vert _{L^{\infty}\left(\Omega^{\varepsilon};C\left(Y_{l}\right)\right)}\left\Vert \varphi_{2}\right\Vert _{L^{2}\left(\Omega^{\varepsilon}\right)}\\
& +C\varepsilon\left\Vert 1-m^{\varepsilon}\right\Vert _{L^{2}\left(\Omega^{\varepsilon}\right)}\left\Vert \nabla_{y}\tilde{\Phi}_{0}\right\Vert _{L^{\infty}\left(\Omega^{\varepsilon};C\left(Y_{l}\right)\right)}\left\Vert \nabla\varphi_{2}\right\Vert _{L^{2}\left(\Omega^{\varepsilon}\right)},
\end{align*}
where we assume that $\tilde{\Phi}_{0}\in L^{\infty}\left(\Omega^{\varepsilon};W^{2+s,2}\left(Y_{l}\right)\right)\cap H^{1}\left(\Omega^{\varepsilon};W^{1+s,2}\left(Y_{l}\right)\right)$
and make use of the compact embeddings $W^{2+s,2}\left(Y_{l}\right)\subset C^{2}\left(Y_{l}\right)$,
$W^{1+s,2}\left(Y_{l}\right)\subset C^{1}\left(Y_{l}\right)$ for
$s>d/2$.
Applying the inequalities (\ref{eq:cutin}), we thus have
\begin{equation}
	\left|\mathcal{F}_{1}\right|+\left|\mathcal{F}_{2}\right|\le C\left(\varepsilon+\varepsilon^{\frac{1}{2}}\right)\left\Vert \varphi_{2}\right\Vert _{L^{2}\left(\Omega^{\varepsilon}\right)}+C\varepsilon^{\frac{3}{2}}\left\Vert \nabla\varphi_{2}\right\Vert _{L^{2}\left(\Omega^{\varepsilon}\right)}.\label{eq:3.50-6}
\end{equation}
It now remains to estimate the following integral:
\[
\int_{\Omega^{\varepsilon}}\varepsilon^{2}\nabla\tilde{\Phi}_{\varepsilon}\cdot\nabla\varphi_{2}dx=\int_{\Omega^{\varepsilon}}\varepsilon\nabla\tilde{\Phi}_{\varepsilon}\cdot\varepsilon\nabla\left(\tilde{\Phi}_{\varepsilon}-\tilde{\Phi}_{0}^{\varepsilon}\right)dx.
\]
Its right-hand side can be estimated by
\begin{equation}
	\int_{\Omega^{\varepsilon}}\varepsilon\nabla\tilde{\Phi}_{\varepsilon}\cdot\varepsilon\nabla\left(\tilde{\Phi}_{\varepsilon}-\tilde{\Phi}_{0}^{\varepsilon}\right)dx\le C\varepsilon\left\Vert \nabla\left(\tilde{\Phi}_{\varepsilon}-\tilde{\Phi}_{0}^{\varepsilon}\right)\right\Vert _{\left[L^{2}\left(\Omega^{\varepsilon}\right)\right]^{d}},\label{eq:3.51-6}
\end{equation}
where we use the fact that $\varepsilon\left\Vert \nabla\tilde{\Phi}_{\varepsilon}\right\Vert _{L^{2}\left(\Omega^{\varepsilon}\right)}\le C$
in Theorem \ref{thm:aprioriestimate-D}.

Based on the corrector estimates for the concentration fields
$c_{\varepsilon}^{\pm}$, we see that
\begin{equation}
	\int_{\Omega^{\varepsilon}}\left[\left(c_{\varepsilon}^{+}-c_{0}^{+}\right)+\left(c_{0}^{-}-c_{\varepsilon}^{-}\right)\right]\varphi_{2}dx\le C\left\Vert c_{\varepsilon}^{\pm}-c_{0}^{\pm}\right\Vert _{L^{2}\left(\Omega^{\varepsilon}\right)}\left\Vert \tilde{\Phi}_{\varepsilon}-\tilde{\Phi}_{0}^{\varepsilon}\right\Vert _{L^{2}\left(\Omega^{\varepsilon}\right)}.\label{eq:3.52-6}
\end{equation}

Setting
\begin{align*}
	w_{1}\left(t\right) & :=\left\Vert \tilde{\Phi}_{\varepsilon}\left(t\right)-\tilde{\Phi}_{0}^{\varepsilon}\left(t\right)\right\Vert _{L^{2}\left(\Omega^{\varepsilon}\right)}^{2}+\left\Vert c_{\varepsilon}^{\pm}\left(t\right)-c_{0}^{\pm}\left(t\right)\right\Vert _{L^{2}\left(\Omega^{\varepsilon}\right)}^{2},\\
	w_{2}\left(t\right) & :=\left\Vert \nabla\left(\tilde{\Phi}_{\varepsilon}-\tilde{\Phi}_{0}^{\varepsilon}\right)\left(t\right)\right\Vert _{\left[L^{2}\left(\Omega^{\varepsilon}\right)\right]^{d}}^{2}+\left\Vert \nabla\left(c_{\varepsilon}^{\pm}-c_{1}^{\pm,\varepsilon}\right)\left(t\right)\right\Vert _{\left[L^{2}\left(\Omega^{\varepsilon}\right)\right]^{d}}^{2},\\
	w_{0} & :=\left\Vert c_{\varepsilon}^{\pm,0}-c_{0}^{\pm,0}\right\Vert _{L^{2}\left(\Omega^{\varepsilon}\right)}^{2},
\end{align*}
the combination of the estimates (\ref{eq:3.50-6})-(\ref{eq:3.52-6})
with the respective estimates for the concentration fields (which are similar
to the Neumann case) and the application of suitable H\"older-like inequalities
give
\[
w_{1}\left(t\right)+\int_{0}^{t}w_{2}\left(s\right)ds\le C\left(\varepsilon+\left(1+\varepsilon\right)w_{0}+\int_{0}^{t}w_{1}\left(s\right)ds\right).
\]

Using Gronwall's inequality yields
\[
w_{1}\left(t\right)+\int_{0}^{t}w_{2}\left(s\right)ds\le C\left(\varepsilon+\left(1+\varepsilon\right)w_{0}\right).
\]

As a consequence, we obtain
\begin{align*}
	& \left\Vert \tilde{\Phi}_{\varepsilon}-\tilde{\Phi}_{0}^{\varepsilon}\right\Vert _{L^{2}\left(\left(0,T\right)\times\Omega^{\varepsilon}\right)}+\left\Vert \nabla\left(\tilde{\Phi}_{\varepsilon}-\tilde{\Phi}_{0}^{\varepsilon}\right)\right\Vert _{\left[L^{2}\left(\left(0,T\right)\times\Omega^{\varepsilon}\right)\right]^{d}}\\
	& +\left\Vert c_{\varepsilon}^{\pm}-c_{0}^{\pm,\varepsilon}\right\Vert _{L^{2}\left(\left(0,T\right)\times\Omega^{\varepsilon}\right)}+\left\Vert \nabla\left(c_{\varepsilon}^{\pm}-c_{1}^{\pm,\varepsilon}\right)\right\Vert _{\left[L^{2}\left(\left(0,T\right)\times\Omega^{\varepsilon}\right)\right]^{d}}\le C\max\left\{ \varepsilon^{\frac{1}{2}},\varepsilon^{\frac{\mu}{2}}\right\} \;\text{for }\mu\in\mathbb{R}_{+},
\end{align*}
where we have used \eqref{initialguess-6}.

Finally, we apply Lemma \ref{lem:ines1} to get
\begin{align*}
\left\Vert \tilde{\Phi}_{\varepsilon}-\overline{\tilde{\Phi}}_{0}^{\varepsilon}\right\Vert _{L^{2}\left(\left(0,T\right)\times\Omega^{\varepsilon}\right)} & \le\left\Vert \tilde{\Phi}_{\varepsilon}-\tilde{\Phi}_{0}^{\varepsilon}\right\Vert _{L^{2}\left(\left(0,T\right)\times\Omega^{\varepsilon}\right)}+\left\Vert \tilde{\Phi}_{0}^{\varepsilon}-\overline{\tilde{\Phi}}_{0}^{\varepsilon}\right\Vert _{L^{2}\left(\left(0,T\right)\times\Omega^{\varepsilon}\right)}\\
& \le C\max\left\{ \varepsilon^{\frac{1}{2}},\varepsilon^{\frac{\mu}{2}}\right\} .
\end{align*}

This completes the proof of Theorem \ref{thm:main2-6}.




\section{Conclusions}\label{finalsec}
         
In \cite{RMK12}, the two-scale convergence method has discovered  possible macroscopic
structures of a non-stationary SNPP model coupled with various scaling
factors and different boundary conditions. In this paper, we have justified such
homogenization limits by deriving several corrector estimates (cf. Theorem
\ref{thm:main1-6} and Theorem \ref{thm:main2-6}). The techniques
we have presented here are mainly based on the construction of suitable macroscopic reconstructions and on a number of energy-like estimates. The employed methodology is applicable to more complex scenarios, where coupled systems of partial differential equations posed in perforated media are involved.
         
\bibliography{mybib}{}
\bibliographystyle{plain}

\end{document}